\pdfoutput=1

\documentclass[a4paper,12pt]{amsart}
\usepackage[T1]{fontenc}
\usepackage[utf8]{inputenc}

\usepackage{todonotes}
\usepackage{geometry}
\usepackage{caption}
\usepackage{subcaption}

\usepackage{amsthm, amsmath, amsfonts}
\usepackage{mathtools}

\usepackage[foot]{amsaddr}

\usepackage{hyperref}

\newcommand{\cmax}[1]{c_{\max}^{(#1)}}

\newcommand{\Z}{\mathbb{Z}}

\newcommand{\Oh}{\mathcal{O}}
\newcommand{\Troot}{\mathcal{T}^{\bullet}}

\newcommand{\E}{\mathbb{E}}
\newcommand{\V}{\mathbb{V}}

\renewcommand{\P}{\mathbb{P}}

\DeclareMathOperator{\Cov}{Cov}
\DeclareMathOperator{\diag}{diag}

\DeclarePairedDelimiter{\abs}{\lvert}{\rvert}
\DeclarePairedDelimiter{\floor}{\lfloor}{\rfloor}
\DeclarePairedDelimiter{\ceil}{\lceil}{\rceil}

\newif\ifdetails
\detailstrue
\newcommand{\DETAIL}[1]%
{\ifdetails\par\fbox{\begin{minipage}{0.9\linewidth}\textit{Detail:}
      #1\end{minipage}}\par\fi}
\newcommand{\TODO}[1]%
{\ifdetails\par\fbox{\begin{minipage}{0.9\linewidth}\textbf{TODO:}
      #1\end{minipage}}\par\fi}

\newtheorem{lemma}{Lemma}

\newtheorem{theorem}[lemma]{Theorem}
\newtheorem{corollary}[lemma]{Corollary}

\theoremstyle{remark}
\newtheorem{remark}{Remark}

\newtheorem{definition}[remark]{Definition}

\title{The uncover process for random labeled trees}

\author{Benjamin Hackl}
\address[Benjamin Hackl]{
   University of Graz, Austria;
   University of Klagenfurt, Austria;
   Uppsala University, Sweden
}
\email{math@benjamin-hackl.at}

\author{Alois Panholzer}
\address[Alois Panholzer]{
   TU Wien, Austria
}
\email{alois.panholzer@tuwien.ac.at}

\author{Stephan Wagner}
\address[Stephan Wagner]{
   Uppsala University, Sweden
}
\email{stephan.wagner@math.uu.se}

\thanks{The third author was supported by the Knut and Alice Wallenberg Foundation. \\
An extended abstract of this paper appeared in the Proceedings of the 33rd International
Conference on Probabilistic, Combinatorial and Asymptotic Methods for the Analysis of Algorithms
(AofA 2022)~\cite{Hackl-Panholzer-Wagner:2022:uncover-extended}.}

\begin{document}

\begin{abstract}
   We consider the process of uncovering the vertices of a random labeled tree according to their labels. First, a labeled tree with $n$ vertices is generated uniformly at random. Thereafter, the vertices are uncovered one by one, in order of their labels. With each new vertex, all edges to previously uncovered vertices are uncovered as well. In this way, one obtains a growing sequence of forests. Three particular aspects of this process are studied in this work: first the number of edges, which we prove to converge to a stochastic process akin to a Brownian bridge after appropriate rescaling. Second, the connected component of a fixed vertex, for which different phases are identified and limiting distributions determined in each phase. Lastly, the largest connected component, for which we also observe a phase transition.
\end{abstract}

\maketitle

\section{Introduction}

We consider the process of uncovering the vertices of a random tree: starting either from
one of the $n^{n-2}$ unrooted or one of the $n^{n-1}$ rooted unordered labeled trees of
size $n$ (i.e., with $n$ vertices) chosen uniformly at random, we uncover the vertices one by one in order of their labels. This yields a growing sequence of forests induced by the uncovered vertices, and we are interested in the evolution of these forests from the first vertex to the point that all vertices are uncovered.

This model is motivated by stochastic models known as coalescent models for particle coalescence,
most notably the additive and the multiplicative coalescent~\cite{Bertoin:2006:random-fragmentation-coagulation} and the Kingman coalescent~\cite{Kingman:1982:coalescent}. To make the distinction between these classical coalescent models
and our model more explicit, let us briefly revisit the additive coalescent
model (see~\cite{Aldous-Pitman:1998:additive-coalescent}) as a prototypical example.
This model describes a Markov process on a state space
consisting of tuples $(x_1, x_2, \dots)$ with $x_1 \geq x_2 \geq \dots \geq 0$
and $\sum_{i\geq 0} x_i = 1$ that model the fragmentation of a unit mass into
clusters of mass $x_i$. Pairs of clusters with masses $x_i$ and $x_j$ then merge
into a new cluster of mass $x_i + x_j$ at rate $x_i + x_j$. In the corresponding
discrete time version of the process, exactly two clusters are merged in every
time step. There are various combinatorial settings in which this discrete
additive coalescent model appears, for example in the evolution of parking
blocks in parking schemes related to ``hashing with linear probing'' \cite{Chassaing-Louchard:2002:parking-blocks}
and in a certain scheme for merging forests by uncovering one edge
in every time step \cite{Pitman:1999:coalescent-random-forests}. There is also a rich literature on random (edge) cutting of trees, starting with the work of Meir and Moon \cite{Meir-Moon:1970:cutting}, see also \cite{Addario-Broutin-Holmgren:2014:cutting,Fill-Kapur-Panholzer:2006:destruction,Janson:2006:random-cutting,Kuba-Panholzer:2008:isolating}. Fragmentation processes on trees have also been studied extensively in a continuous setting, see e.g.~\cite{Brezunza-Holmgren-2020:invariance,Miermont:2003:self-similar,Miermont:2005:self-similar,Thevenin:2021:geometric}. Lastly, the special case of our model in which the underlying tree is always a path (with random labels) rather than a random labeled tree was considered by Janson in \cite{Janson:2009:sorting}.

While the incarnation of the additive coalescent in which edges are uncovered
successively is very much related in spirit to our vertex uncover model, the
underlying processes are fundamentally different: these classical coalescent models
rely on the fact that exactly two clusters are merged in every time step, which
is not the case in our model. When uncovering a new vertex, a more or less arbitrary
number of edges (including none at all) can be uncovered. There are coalescent
models like the $\Lambda$-coalescent, a generalization of the Kingman coalescent~\cite{Pitman:1999:multiple-collisions}, which allow for more than two clusters being
merged---however, at present we are not aware of any known coalescent model that is
able to capture the behavior of the vertex uncover process.

\begin{figure}[ht]
   \centering
   \begin{subfigure}[t]{0.32\linewidth}
      \includegraphics[width=\linewidth]{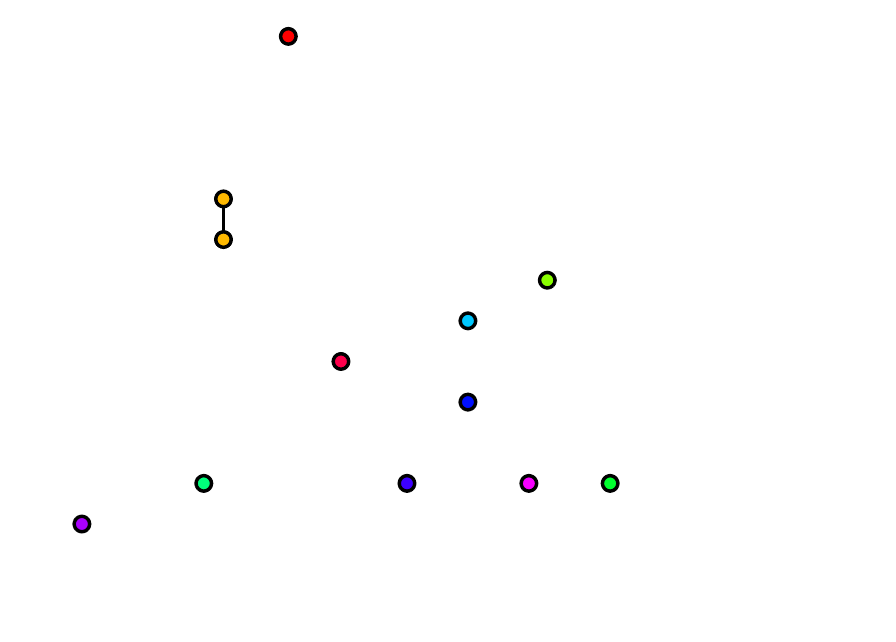}
   \end{subfigure}
   \begin{subfigure}[t]{0.32\linewidth}
      \includegraphics[width=\linewidth]{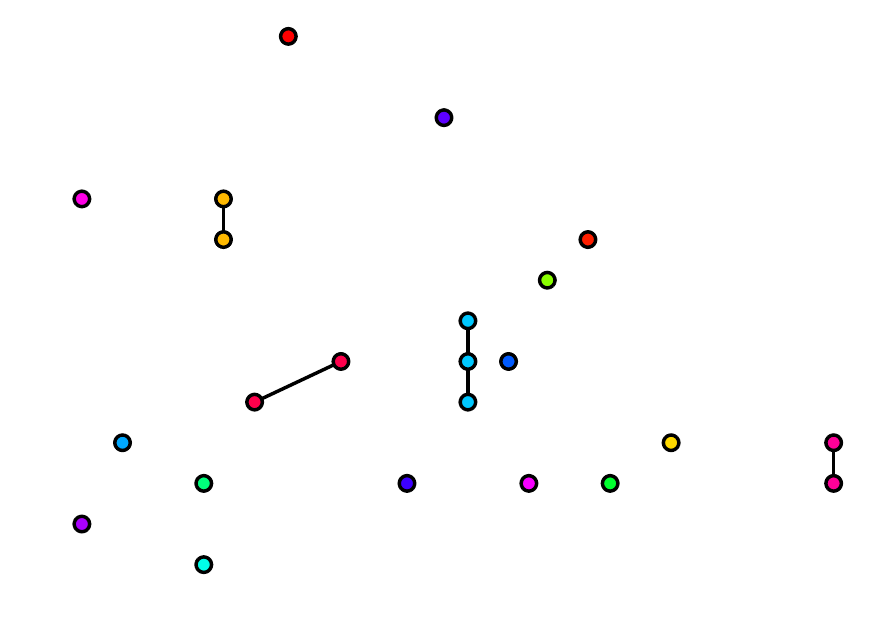}
   \end{subfigure}
   \begin{subfigure}[t]{0.32\linewidth}
      \includegraphics[width=\linewidth]{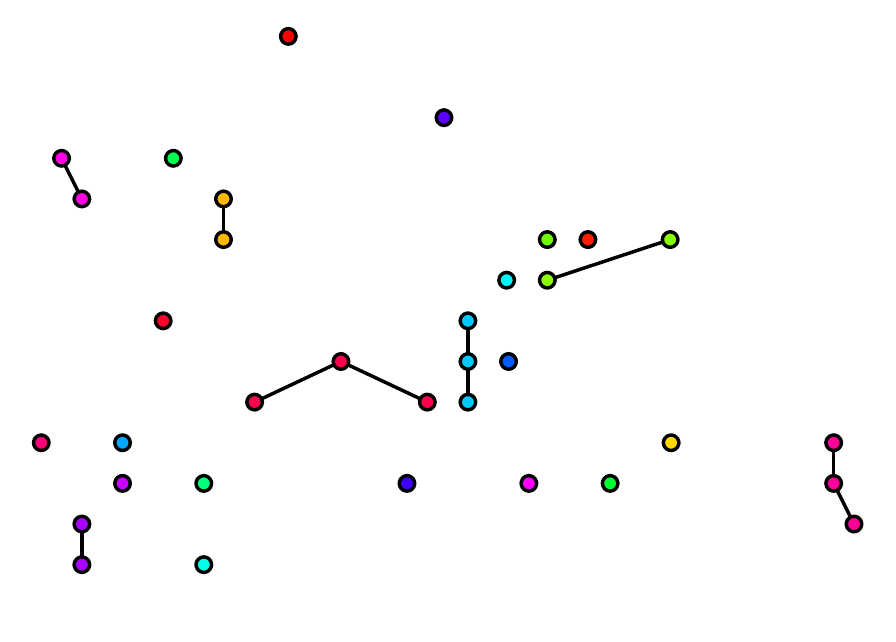}
   \end{subfigure}\\
   \begin{subfigure}[t]{0.32\linewidth}
      \includegraphics[width=\linewidth]{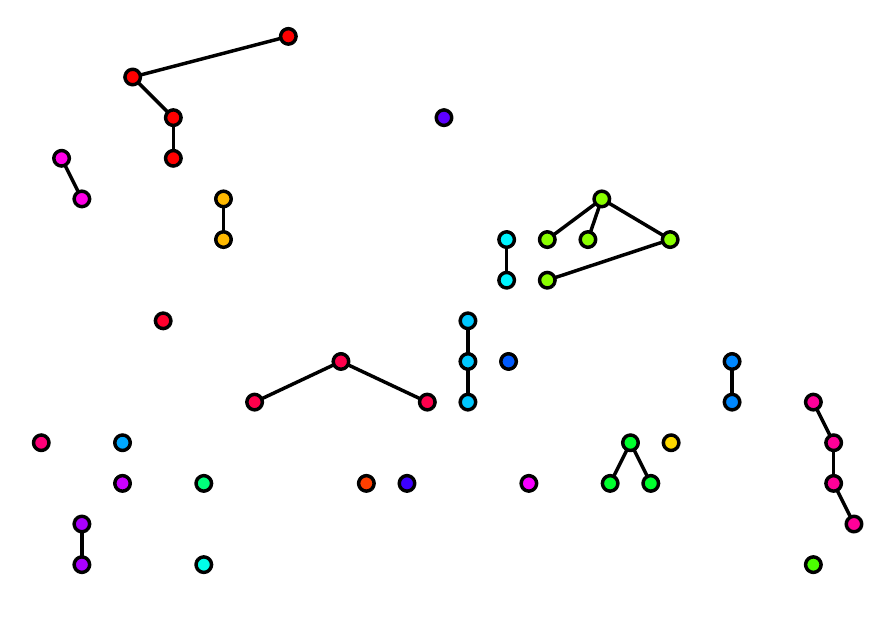}
   \end{subfigure}
   \begin{subfigure}[t]{0.32\linewidth}
      \includegraphics[width=\linewidth]{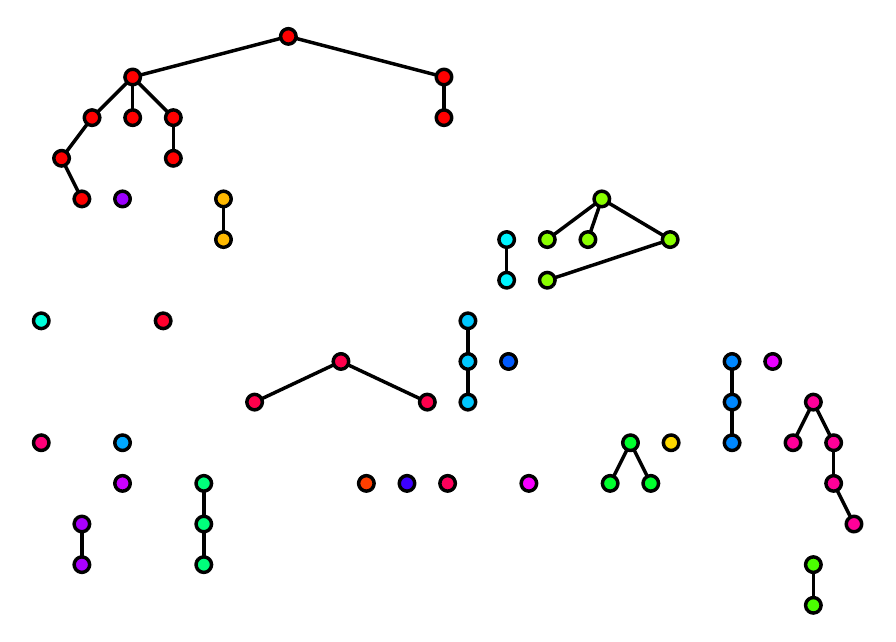}
   \end{subfigure}
   \begin{subfigure}[t]{0.32\linewidth}
      \includegraphics[width=\linewidth]{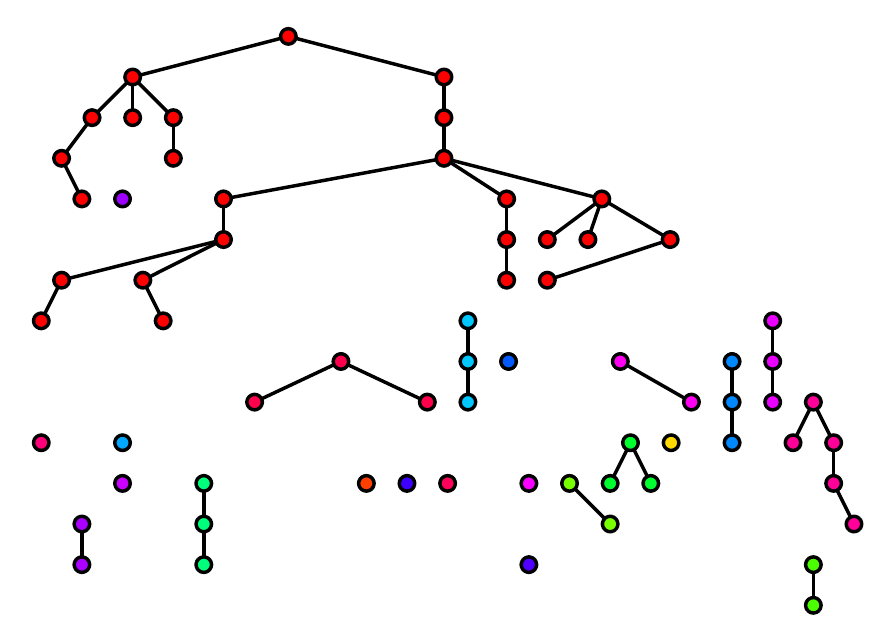}
   \end{subfigure}\\
   \begin{subfigure}[t]{0.32\linewidth}
      \includegraphics[width=\linewidth]{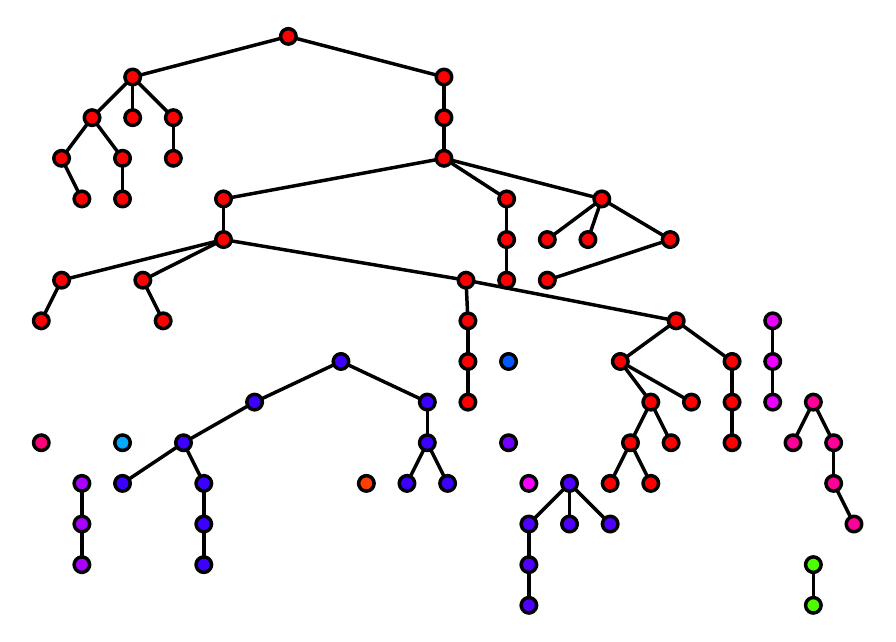}
   \end{subfigure}
   \begin{subfigure}[t]{0.32\linewidth}
      \includegraphics[width=\linewidth]{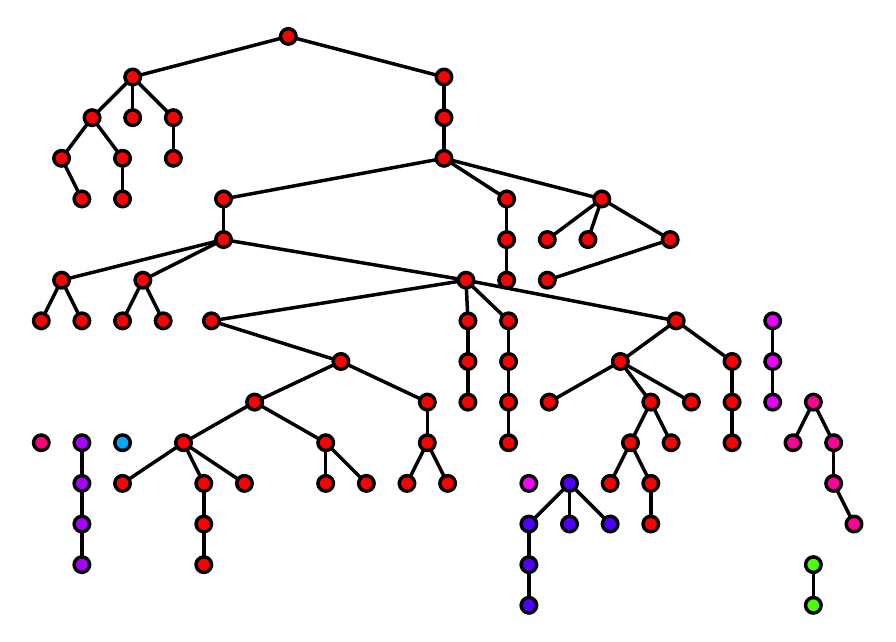}
   \end{subfigure}
   \begin{subfigure}[t]{0.32\linewidth}
      \includegraphics[width=\linewidth]{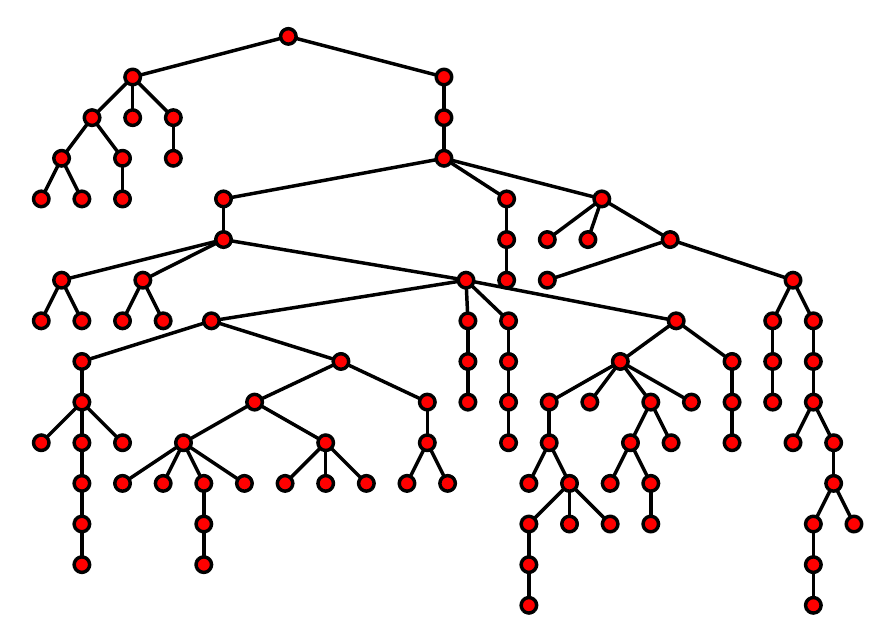}
   \end{subfigure}
   \caption{A few snapshots of the \emph{uncover process} applied to a random
   labeled tree of size 100. From left to right and top to bottom, there are
   12, 23, 34, \dots, 89, and 100 uncovered vertices in the figures, respectively.
   Vertex labels are omitted for the sake of readability, and vertices are colored
   per connected component.
   }
   \label{fig:uncover-example}
\end{figure}

\subparagraph{Overview.} Different aspects of the uncover process on labeled
trees are investigated in this work.
In Section~\ref{sec:edges}, we study the stochastic process
given by the number of uncovered edges. The corresponding main result,
a full characterization of the process and its limiting behavior, is given in
Theorem~\ref{thm:edge-process}.

Sections~\ref{sec:root-cluster} and~\ref{sec:largest-component} are both
concerned with cluster sizes, i.e., with the sizes of the connected components
that are created throughout the process. In particular, in Section~\ref{sec:root-cluster}
we shift our attention to rooted labeled trees, to study the behavior of the
component containing a fixed vertex. The expected size of the root cluster
is analyzed in Theorem~\ref{thm:R_Exp_labeled}. Furthermore, we show
that the number of rooted trees whose root cluster has a given size is given
by a rather simple enumeration formula---which, in turn, manifests
in Theorem~\ref{thm:R_Prob_labeled}, a characterization of the different
limiting distributions for the root cluster size depending on the number of
uncovered vertices.

Finally, in Section~\ref{sec:largest-component} we use the results
on the root cluster to draw conclusions regarding the size of the
largest cluster in the tree.

\subparagraph{Notation.} Throughout this work we use the notation $[n] = \{1,\dots,n\}$ and $[k, \ell] = \{k, k+1, \dots, \ell\}$ for discrete intervals, and $x^{\underline{j}} = x (x-1) \cdots (x-j+1)$ for the falling factorials. The floor and ceiling function are denoted by $\lfloor x \rfloor$ and $\lceil x \rceil$, respectively. Furthermore, we use $\mathcal{T}$ and $\Troot$ for the combinatorial classes of labeled trees and rooted labeled trees, respectively, and $\mathcal{T}_n$ and $\Troot_n$ for the classes of labeled and rooted labeled trees of size $n$, i.e., with $n$ vertices. Finally, we use $X_n \xrightarrow{d} X$ and $X_n \xrightarrow{p} X$ to denote convergence in distribution resp.~probability of a sequence of random variables (r.v.) $(X_{n})_{n\geq 0}$ to the r.v.\ $X$.

\section{The number of uncovered edges}\label{sec:edges}

In this section our main interest is the behavior of the number
of uncovered edges in the uncover process.
We begin by introducing a formal parameter for this quantity.

\begin{definition}
   Let $T$ be a labeled tree with vertex set $V(T) = [n]$. For $1\leq j \leq n$,
   we let $k_j(T) \coloneqq \| T[1, 2, \dots, j] \|$
   denote the number of edges in the subgraph of $T$ induced by the vertices
   in $[j]$. We refer to the sequence
   $(k_j(T))_{1\leq j\leq n}$ as the \emph{(edge) uncover sequence}.
\end{definition}

We start with a few simple observations.
First, for any labeled tree of order $n$ we have $k_1(T) = 0$,
as well as $k_n(T) = n-1$. Second, as any induced subgraph of a tree is
a forest, and as forests have the elementary property that the number of
edges together with the number of connected components gives the order of
the forest, we find that $j - k_j(T)$ is the number of connected
components after uncovering the first $j$ vertices of $T$.
Figure~\ref{fig:labeled-tree:example} illustrates the progression
of the number of edges and the number of connected components
for $1 \leq j\leq 1000$ in a randomly chosen labeled
tree on 1000 vertices.

\begin{figure}[ht]
   \centering
   \begin{subfigure}[t]{0.48\linewidth}
      \centering
      \includegraphics[width=\linewidth]{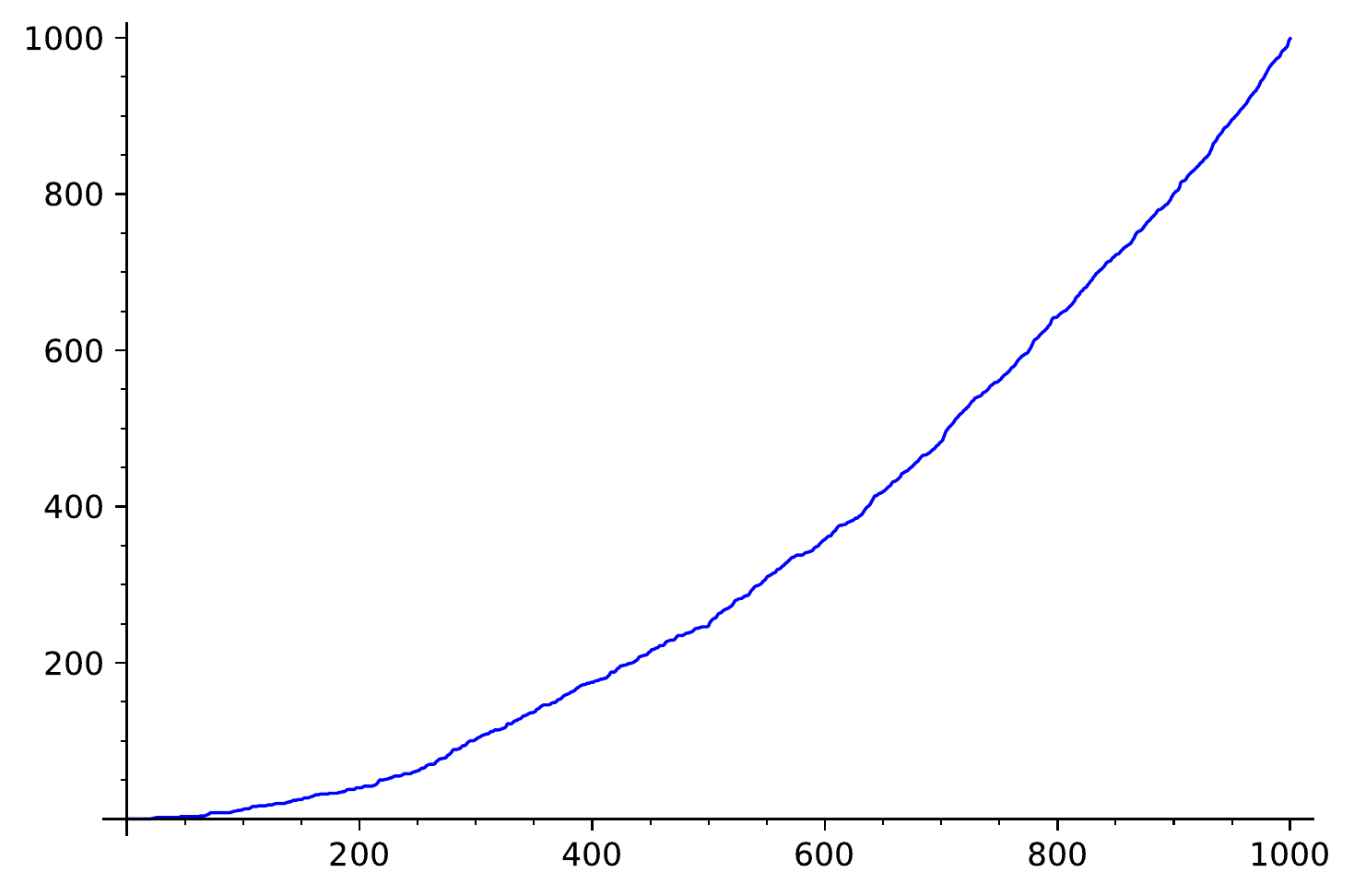}
   \end{subfigure}
   \hfill
   \begin{subfigure}[t]{0.48\linewidth}
      \centering
      \includegraphics[width=\linewidth]{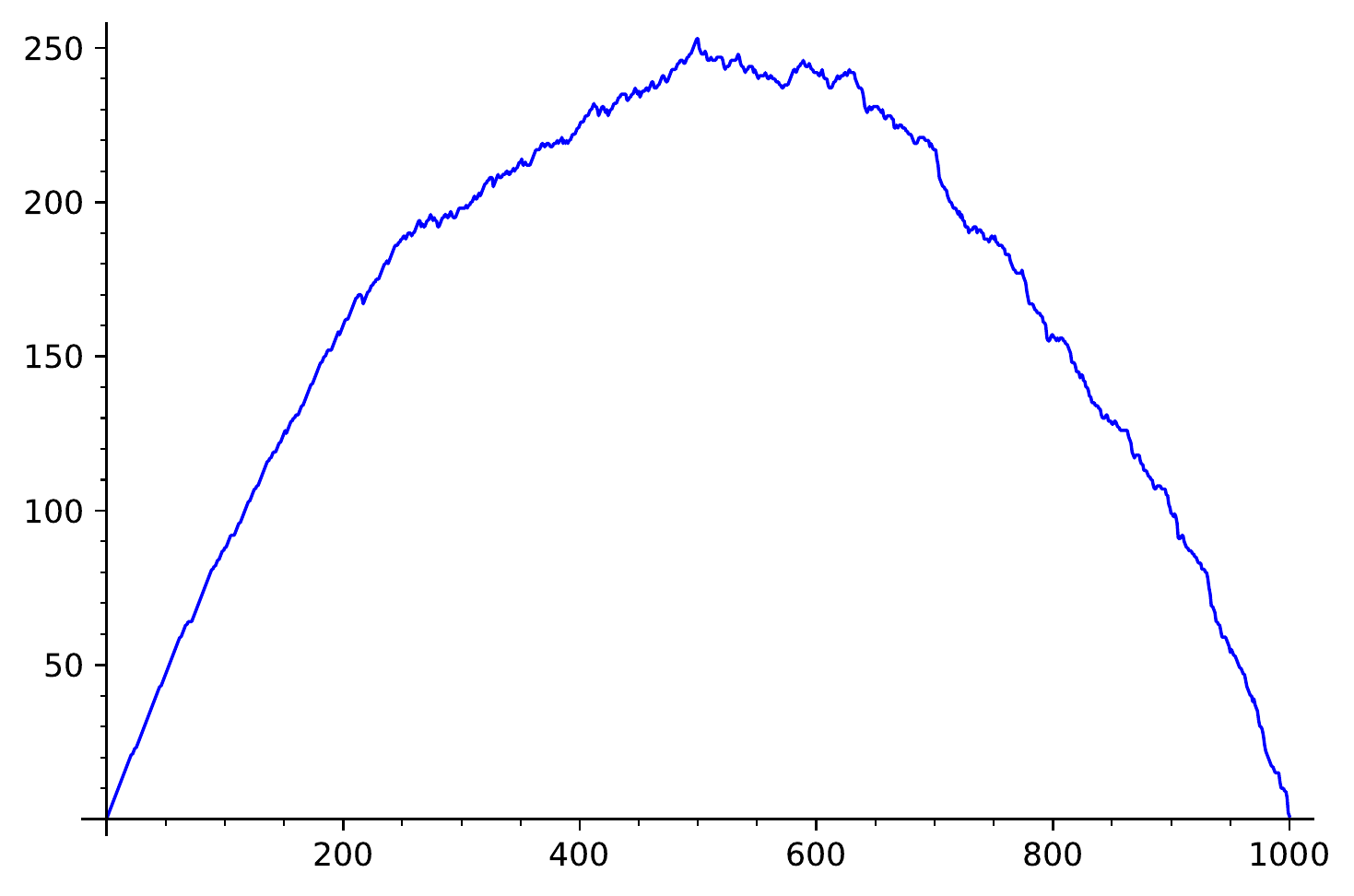}
   \end{subfigure}
   \caption{Progression of the number of edges (left) and the number
   of connected components (right) when sequentially uncovering a random labeled
   tree on 1000 vertices.}
   \label{fig:labeled-tree:example}
\end{figure}

Moreover, the fact that the first $j$ vertices of the
tree induce a forest also yields
the sharp bound $0 \leq k_j(T) \leq j-1$ for all $1\leq j \leq n-1$.
The lower bound is attained by the star with central vertex $n$,
and the upper bound is attained by the (linearly ordered) path.
We can also observe that as soon as $k_{n-1}(T) > 0$, the set of
edges added in the last uncover step is not determined uniquely.
Thus, the star with central vertex $n$ is the only tree that is
fully determined by its uncover sequence.

The following theorem provides explicit enumeration formulas for
the number of trees with a partially and fully specified uncover
sequence, respectively.

\begin{theorem}\label{thm:edges:explicit}
   Let $r$ be a fixed positive integer with $1\leq r < n-1$, let
   $j_1$, $j_2$, \ldots, $j_r$ be positive integers with $1 < j_1 < j_2 < \cdots < j_r < n$,
   and let $a_1$, $a_2$, \ldots, $a_r$ be a non-decreasing sequence of non-negative integers
   satisfying $a_i \leq j_i - 1$ for all $1\leq i \leq r$.
   Additionally, let $j_0 = 1$ and $a_0 = 0$.
   Then, the number of rooted labeled trees $T$ of order $n$ that satisfy $k_{j_i}(T) = a_i$ for
   all $1\leq i\leq r$ is given by
   \begin{multline}
      (n-j_r)^{j_r-a_r-1}n^{n-j_r-1} \\
      \times \prod_{i=1}^r \bigg( \sum_{h=0}^{a_i-a_{i-1}} \binom{j_{i-1}-a_{i-1}-1}{h}
      \binom{j_i-j_{i-1}}{a_i-a_{i-1}-h} (j_i-j_{i-1})^h j_i^{a_i-a_{i-1}-h} \bigg).
   \end{multline}
   Furthermore, there are
   \begin{equation}
      \prod_{i=1}^{n-2} \biggl(\binom{i - a_i - 1}{a_{i+1} - a_i - 1}(i+1) + \binom{i - a_i - 1}{a_{i+1} - a_i}\biggr)
   \end{equation}
   trees with a fully specified uncover sequence
   $(0, a_2, a_3, \dots, a_{n-1}, n-1)$.
\end{theorem}

We first derive
a helpful auxiliary result, namely an explicit formula for the corresponding
(multivariate) generating function. The enumeration formula will then follow
by extracting the appropriate coefficients.

\begin{lemma}\label{lemma:edges-gf}
   In the setting of Theorem~\ref{thm:edges:explicit},
   the multivariate generating function for the increments
   in the edge uncover process is given by
   \begin{equation}\label{eq:lem:edges-gf}
      E_n(z_1, z_2, \ldots, z_r) = n^{n-j_r-1}
      \prod_{i=1}^{r} \Big( n-j_r + j_iz_i + \sum_{h=i+1}^r (j_h-j_{h-1})z_h \Big)^{j_i - j_{i-1}}.
   \end{equation}
   In other words, the coefficient of the monomial
   $z_1^{a_1} z_2^{a_2 - a_1} \dots z_r^{a_r - a_{r-1}}$ in the expansion of $E_n(z_1, \dots, z_r)$
   is the number of labeled trees $T$ of order $n$ with $k_{j_i}(T) = a_i$ for all $1\leq i \leq r$.
\end{lemma}
\begin{remark}
   By specifying the integers $j_1$, $j_2$, \dots, $j_r$, the uncover process
   is effectively partitioned into intervals. This is also reflected by the quantities
   occurring in the product in~\eqref{eq:lem:edges-gf}: the difference
   $j_i - j_{i-1}$ corresponds to the number of vertices uncovered in the $i$-th
   interval, $j_i$ represents the number of vertices uncovered in total up to the
   $i$-th interval, and $n - j_r$ corresponds to the number of vertices uncovered
   in the last interval.
\end{remark}
\begin{proof}[Proof of Lemma~\ref{lemma:edges-gf}]
   We begin by observing that when the process uncovers the vertex with label~$j$, edges to
   all adjacent vertices whose label is less than $j$ are uncovered as well. To determine
   the generating function of the edge increments, we assign the weight $x_i y_j$
   to the edge connecting vertex $i$ and vertex $j$ with $i < j$, and then consider the
   generating function for the tree weight $w(T)$ (which is defined as the product of the
   edge weights); $E_n(z_1,\dots, z_r) = \sum_{|T| = n} w(T)$.

   Following Martin and Reiner~\cite[Theorem 4]{Martin-Reiner:2003:factorization}
   or Remmel and Williamson~\cite[Equation (8)]{Remmel-Williamson:2002:spanning-trees}, the
   generating function of the tree weights $w(T)$ has the explicit formula
   \begin{equation}\label{eq:edges:gf-general}
      \sum_{|T| = n} w(T) = x_1 y_n \prod_{j=2}^{n-1} \Big( \sum_{i=1}^n x_{\min(i,j)} y_{\max(i,j)} \Big).
   \end{equation}
   As initially observed, edges that are counted by $k_{j_i}(T)$ are precisely those that induce
   a factor $y_{\ell}$ for some
   $\ell \leq j_i$. Thus we make the following substitutions: $x_{\ell} = 1$ for
   all $\ell$, $y_{\ell} = z_i$ if
   and only if $j_{i-1} < \ell \leq j_i$ (where\footnote{Observe that $y_1$ does not occur,
   since at least one of the ends of every edge has a label greater than $1$.} $j_0 = 1$),
   and $y_{\ell} = 1$ if $\ell > j_r$. To deal with the sum over $y_{\max(i,j)}$, observe that
   we can rewrite it as
   \[
      \sum_{i=1}^{n} y_{\max(i, j)} = n - j_r + \sum_{i=1}^{j_1} y_{\max(i, j)} + \cdots
      + \sum_{i=j_{r-1}+1}^{j_r} y_{\max(i, j)}.
   \]
   In this form, the different values assumed by the sum when $j$ moves through the ranges
   $1 < j\leq j_1$, $j_1 < j \leq j_2$, etc. can be determined directly. For some $j$
   with $j_{i-1} < j \leq j_{i}$, the contribution to the product in~\eqref{eq:edges:gf-general} is
   \[ n-j_r + j_iz_i + \sum_{h=i+1}^r (j_h-j_{h-1})z_h, \]
   and for $j_r < j \leq n-1$ all $y$-variables are replaced by 1, so that the
   contribution to the product is a factor $n$. Putting both of these
   observations together shows that the right-hand side of~\eqref{eq:edges:gf-general}
   can be rewritten as the right-hand side of~\eqref{eq:lem:edges-gf} and thus proves
   the lemma.
\end{proof}

With an explicit formula for the generating function of edge increments in the uncover
process at our disposal, an explicit formula for the number of trees with given (partial)
uncover sequence follows as a simple consequence.

\begin{proof}[Proof of Theorem~\ref{thm:edges:explicit}]
   It remains to extract the coefficient of $z_1^{a_1}z_2^{a_2-a_1} \cdots z_r^{a_r-a_{r-1}}$,
   which is done step by step, starting with $z_1$.
\end{proof}

\subsection{A closer look at the stochastic process}\label{sec:edges:process}
The exceptionally nice formula for the generating function of edge increments
can be used to study the stochastic process that describes the number of
uncovered edges in more detail. Let the sequence of random variables
$(K_j^{(n)})_{1\leq j\leq n}$ be the discrete stochastic process
modeling the number of uncovered edges after uncovering the first $j$ vertices
in a random labeled tree of size $n$,
chosen uniformly at random. The expected number of uncovered edges can be
determined by a simple argument: with $j$ uncovered vertices,
$\binom{j}{2}$ of the $\binom{n}{2}$ possible positions for the edges
have been uncovered. As every position is, due to symmetry and the uniform
choice of the labeled tree, equally likely to hold one of the $n-1$
edges, we find
\begin{equation}\label{eq:edges:expectation}
   \E K_j^{(n)} = (n-1) \frac{\binom{j}{2}}{\binom{n}{2}} = \frac{j(j-1)}{n}.
\end{equation}
To motivate our investigations further, consider the illustrations in
Figure~\ref{fig:process:example}. The rescaled deviation from the mean
is reminiscent of a stochastic process known as Brownian bridge.

In order to define this process formally, recall first that the
Wiener process $(W(t))_{t\in[0,1]}$ is the unique stochastic
process that satisfies
\begin{itemize}
\item $W(0) = 0$,
\item $W$ has independent, stationary increments,
\item $W(t) \sim \mathcal{N}(0, t)$ for all $t > 0$, and
\item $t\mapsto W(t)$ is almost surely continuous,
\end{itemize}
see~\cite[Definition 21.8]{Klenke:2020:probability-theory-course}. A Brownian
bridge can then be defined by setting
\begin{equation}\label{eq:brownian-bridge}
   B(t) = (1-t) W(t/(1-t)),
\end{equation}
see e.g.~\cite[Exercise 3.10]{Revuz-Yor:1999:continuous-martingales-bm}.
The term ``bridge'' results from the fact that we have $B(0) = B(1) = 0$.

While the (normalized) deviation
from the mean looks like it might converge to a Brownian
bridge, we will prove that this is only \emph{almost} the case. The
following theorem characterizes the stochastic process.
For technical purposes, we set $K_0^{(n)} = 0$ and introduce the linearly interpolated process $(\tilde K_t^{(n)})_{t\in[0,1]}$,
where
\begin{equation}\label{eq:interpolation}
   \tilde K_t^{(n)} :=  (1 + \floor{tn} - tn) K_{\floor{tn}}^{(n)} + (tn - \floor{tn}) K_{\ceil{tn}}^{(n)},
\end{equation}
which by construction has continuous paths.

\begin{figure}[ht]
   \centering
   \begin{subfigure}[t]{0.48\linewidth}
      \centering
      \includegraphics[width=\linewidth]{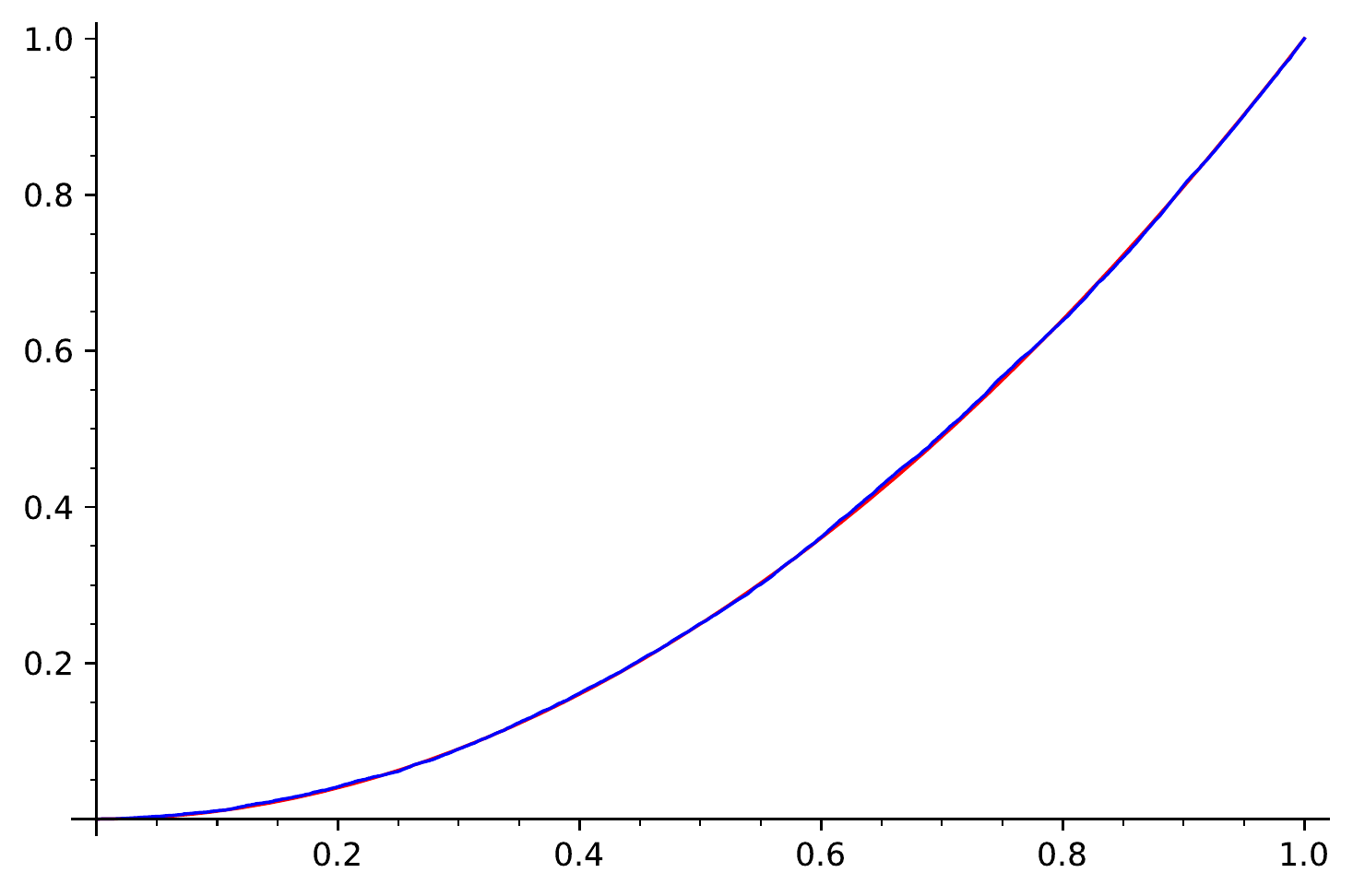}
   \end{subfigure}
   \hfill
   \begin{subfigure}[t]{0.48\linewidth}
      \centering
      \includegraphics[width=\linewidth]{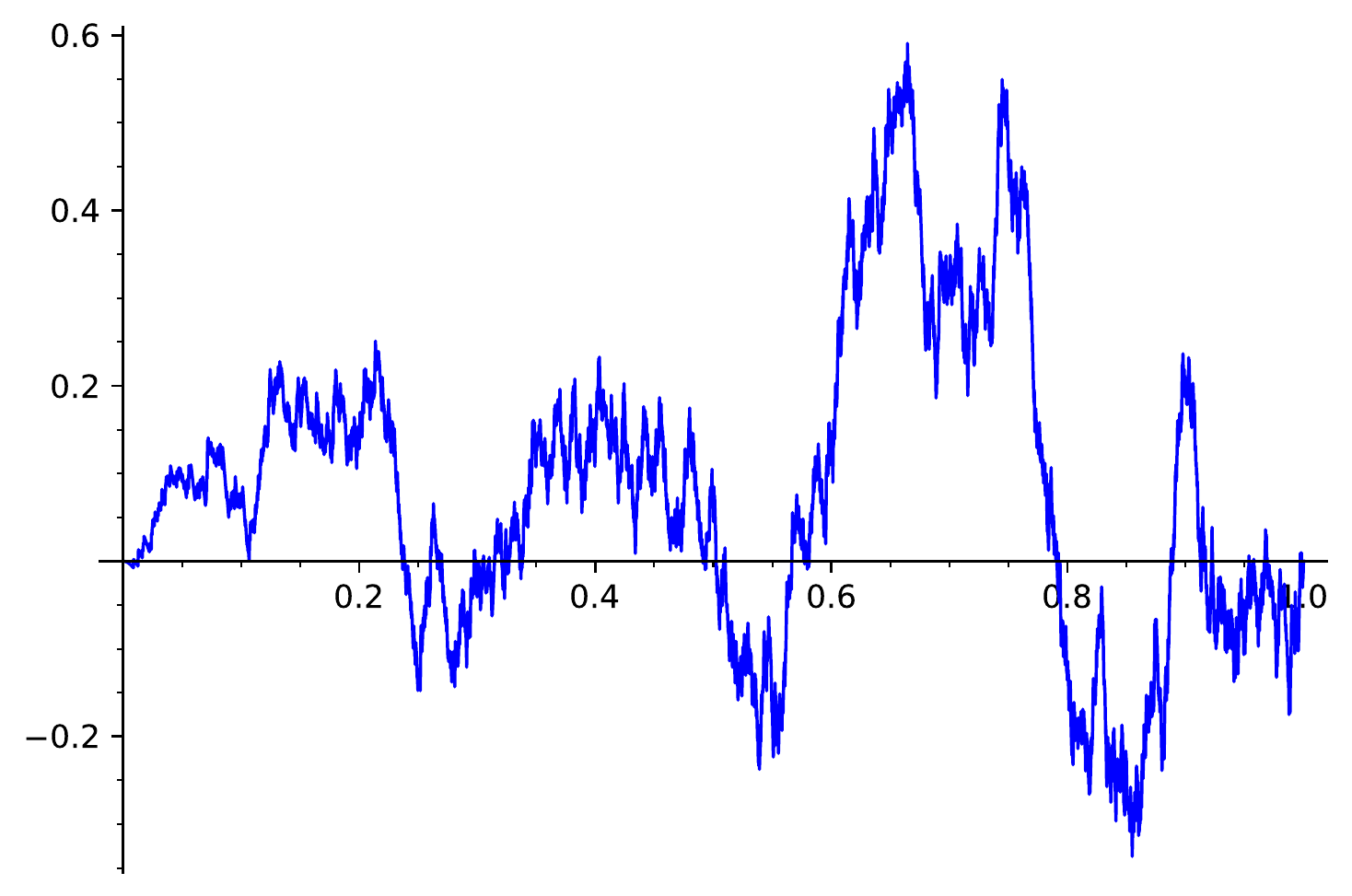}
   \end{subfigure}
   \caption{
      A path of the rescaled stochastic process $(K_{\floor{tn}}^{(n)}/(n-1))_{t\in[0,1]}$
      (left-hand side) and the corresponding (rescaled) deviation
      $\big(\frac{\tilde K_{t}^{(n)} - t^2 n}{\sqrt{n}}\big)_{t\in[0,1]}$ for
      a random labeled tree of size $n = 10000$.
   }
   \label{fig:process:example}
\end{figure}

\begin{theorem}\label{thm:edge-process}
   Let $(Z^{(n)}(t))_{t\in [0,1]}$ be the continuous stochastic process resulting from centering and rescaling the linearly interpolated process
   $(\tilde K_{t}^{(n)})_{t\in[0,1]}$ in the form of
   \[
      Z^{(n)}(t) \coloneqq \frac{\tilde K_{t}^{(n)} - t^2 n}{\sqrt{n}},
   \]
   and let $(W(t))_{t\in[0,1]}$ be the standard Wiener process. Then, for $n\to\infty$, the rescaled
   process converges weakly with respect to the $\sup$-norm on $C([0,1])$ to a limiting process $Z^{\infty}(t)$ that is given by
   \begin{equation}\label{eq:edge-process:limit}
      Z^{\infty}(t) = (1-t) W\big(t^2 / (1-t)\big).
   \end{equation}
   Furthermore, for $s$, $t\in[0,1]$ with $s < t$, the limiting process satisfies
   \begin{equation}\label{eq:edge-process:var-cov}
      \E Z^{\infty}(t) = 0,\quad \V Z^{\infty}(t) = t^2 (1 - t), \quad\text{ and }\quad
      \Cov(Z^{\infty}(s), Z^{\infty}(t)) = s^2 (1 - t).
   \end{equation}
\end{theorem}
\begin{remark}
   While the limiting process $(Z^{\infty}(t))_{t\in[0,1]}$ is not a Brownian bridge
   (the corresponding variances and covariances as given
   in~\eqref{eq:edge-process:var-cov} do not match), it is closely related.
   Comparing the characterization of $Z^{\infty}(t)$
   in~\eqref{eq:edge-process:limit} to~\eqref{eq:brownian-bridge}, we see that the processes only
   differ by the square in the numerator of the argument of the Wiener process.
\end{remark}

Two main ingredients are required to prove this result
(cf.~\cite[Theorem 21.38]{Klenke:2020:probability-theory-course}):
the fact that the sequence of stochastic processes is \emph{tight}
on the one hand, and information
on the finite-dimensional joint distributions
of $(\tilde K_t^{(n)})_{t\in[0,1]}$ on the other hand.

By Prohorov's theorem (\cite[Theorem 13.29]{Klenke:2020:probability-theory-course}),
tightness is equivalent to the sequence being weakly relatively sequentially compact.
We prove that this is the case by checking Kolmogorov's criterion
\cite[Theorem 21.42]{Klenke:2020:probability-theory-course} for which we
have to verify that the family of initial distributions
$(\tilde Z^{(n)}(0))_{n\in \Z_{\geq 0}}$ is tight, and that the paths
of $(Z^{(n)}(t))_{t\in[0,1]}$ cannot change too fast.

While tightness of the initial distributions follows in a rather straightforward
way (given that our process is deterministic in the beginning), we can even derive
a much stronger, uniform bound. Let us begin by revisiting~\eqref{eq:lem:edges-gf}.
Given Cayley's well-known tree enumeration formula, the corresponding
probability generating function for the complete uncover sequence, i.e.,
when we choose our integer vector as $\mathbf{j} = (2, 3, \dots, n-1)$,
is
\begin{equation}\label{eq:edge-uncover:pgf}
P_n(z_2, \dots, z_{n-1}) =
\prod_{i=2}^{n-1} \Bigl(\frac{1}{n} + \frac{i}{n}z_i + \sum_{h=i+1}^{n-1}\frac{1}{n} z_h\Bigr).
\end{equation}
This suggests modeling the process with $n-2$ independent random variables,
each representing an edge increment\footnote{We explicitly model edge increments here
instead of edges, because with this approach we do not need to care about \emph{which}
edge is being uncovered. Our model explicitly only captures the behavior of the number of
uncovered edges.}. The factorization suggests that the
$j$-th increment (which corresponds to the factor with $i = j+1$) happens
with probability $(j+1)/n$ when the vertex
with label $j+1$ is uncovered, or with probability $1/n$ every time
any of the subsequent vertices are uncovered. This probabilistic point of
view can be used to construct a recursive characterization for
the number of uncovered edges, namely\footnote{We slightly abuse notation:
formally, we would need to introduce auxiliary variables that are distributed
according to the specified binomial and Bernoulli distributions.}
\begin{equation}\label{eq:edge-uncover:recursive}
   K_{j+1}^{(n)} = K_{j}^{(n)} + \operatorname{Ber}\Bigl(\frac{j+1}{n}\Bigr) + \operatorname{Bin}\Bigl(j - 1  - K_j^{(n)}, \frac{1}{n-j}\Bigr).
\end{equation}
The Bernoulli variable models the probability that the $j$-th edge increment is
added when uncovering the vertex with label $j+1$, and the binomial variable
models all of the remaining, not yet uncovered edge increments.

Now let us consider a centered and rescaled version of the process
$(K_j^{(n)})_{1\leq j\leq n}$ by defining%
\begin{equation}\label{eq:rescaled-martingale}
   Y_j^{(n)} \coloneqq \frac{K_j^{(n)} - \frac{j(j-1)}{n}}{n-j}.
\end{equation}
With the help of the recursive description in~\eqref{eq:edge-uncover:recursive}, we can
show that $(Y_j^{(n)})_{1\leq j\leq n-1}$ is a martingale by computing
\begin{align*}
   \E(Y_{j+1}^{(n)} | Y_j^{(n)})
   &= \frac{\E(K_{j+1}^{(n)} | K_j^{(n)})}{n-j-1} - \frac{j(j+1)}{n (n-j-1)}= \frac{K_j^{(n)} + \frac{j+1}{n} + \frac{j-1-K_j^{(n)}}{n-j}}{n-j-1} - \frac{j(j+1)}{n (n-j-1)}\\
   &= \frac{K_j^{(n)}}{n-j} - \frac{(j-1)j}{n(n-j)} = Y_j^{(n)}.
\end{align*}
We can also give an explicit expression for the variance of $Y_k^{(n)}$: recall that by~\eqref{eq:rescaled-martingale}, we have $\V Y_k^{(n)} = (n-k)^{-2} \V K_k^{(n)}$.
Then, with~\eqref{eq:edge-uncover:recursive} and the laws of total variance and total
expectation we find the recurrence
\[ \V K_{k+1}^{(n)} = \Bigl(1 - \frac{1}{n - k}\Bigr)^{2} \V K_k^{(n)} + \frac{(n-k-1)(2n-k-1)k}{(n-k)n^2}, \]
for $1 \leq k < n-1$ and $\V K_1^{(n)} = 0$. This allows us to conclude that
\begin{align}
   \V K_{k}^{(n)}
   &= \sum_{j=1}^{k-1} \Bigl(\frac{n-k}{n-j-1}\Bigr)^2 \frac{(n-j-1)(2n-j-1)j}{(n-j)n^2} = \frac{k(k-1)(n-k)}{n^2}, \label{eq:uncover-variance-explicit}
\end{align}
where the sum can be evaluated with the help of partial fractions and telescoping.

The uniform bound (that also implies tightness of $Z^{(n)}(t)$ for every fixed $t$)
can now be stated as follows.
\begin{lemma}\label{lem:edge-process:tight}
   For any real $C > 1$ and any positive integer $n$, the random variable
   $Z^{(n)}(t)$ satisfies the bound
   \begin{equation}\label{eq:edge-process:tight}
      \P(\sup_{t\in[0,1]} \abs{Z^{(n)}(t)} \geq C) \leq 4 (C-1)^{-2},
   \end{equation}
   so that for $C \to\infty$, the probability for the process to
   exceed $C$ in absolute value converges to 0 uniformly in terms of $n$.
\end{lemma}
\begin{proof}
   In order to obtain this condition, we show first that it can be reduced to an inequality for the
   martingale from the previous section. To this end, let us write $tn = j + \eta$, with $j \in \Z$ and $\eta \in [0,1)$.
   A simple calculation shows that
   \begin{align*}
      Z^{(n)}(t) &= \frac{\tilde K_{t}^{(n)} - t^2 n}{\sqrt{n}} \\
      &= \frac{(1-\eta)K_j^{(n)} + \eta K_{j+1}^{(n)} - (j+\eta)^2/n}{\sqrt{n}} \\
      &= \frac{(1-\eta)(K_j^{(n)} - j(j-1)/n) + \eta(K_{j+1}^{(n)} - j(j+1)/n) - (j+\eta^2)/n}{\sqrt{n}} \\
      &= (1-\eta)\frac{K_j^{(n)} - j(j-1)/n}{\sqrt{n}} + \eta \frac{K_{j+1}^{(n)} - j(j+1)/n}{\sqrt{n}}
      - \frac{j+\eta^2}{n^{3/2}}.
   \end{align*}
   The final fraction is bounded by $1$, since $j+\eta^2 \leq j+\eta = tn \leq n$. It follows that
   \[
      \sup_{t\in[0,1]} \abs{Z^{(n)}(t)} \leq \sup_{0 \leq j \leq n}
        \abs[\Big]{\frac{K_j^{(n)} - j(j-1)/n}{\sqrt{n}}} + 1,
   \]
so
   \begin{align}
      \P\bigl(\sup_{t\in[0,1]} \abs{Z^{(n)}(t)} \geq C\bigr) &\leq
      \P\Bigl( \sup_{0 \leq j \leq n} \abs[\Big]{\frac{K_j^{(n)} - j(j-1)/n}{\sqrt{n}}} \geq C-1 \Bigr) \nonumber \\
      &= \P\Bigl( \sup_{1 \leq j \leq n-1} \abs[\Big]{\frac{Y_j^{(n)} (n - j)}{\sqrt{n}}} \geq C-1\Bigr). \label{eq:supineq}
   \end{align}
   Note here that we need not consider $j=0$ and $j=n$ in the supremum, since $K_j^{(n)} - j(j-1)/n = 0$ in either case.
   Since $(Y_j^{(n)})_{1\leq j\leq n-1}$ is a martingale, we can use Doob's
   $L^p$-inequality~\cite[Theorem 11.2]{Klenke:2020:probability-theory-course}.
   For any real $C > 0$ and any fixed integer $k$ with $1\leq k\leq n-1$, we have
   \[
      \P\bigl(\sup_{1\leq j\leq k} \abs{Y_j^{(n)}} \geq C \bigr) \leq \frac{\V Y_k^{(n)}}{C^2} = \frac{k (k-1)}{C^2 (n-k) n^2}.
   \]
   With this, we have all required prerequisites to prove the bound. We partition the interval over
   which the supremum is taken in~\eqref{eq:supineq},
   apply the martingale inequality, and then
   obtain the desired result after summing over all these upper bounds.
   For every integer $i > 0$, let $I_i^{(n)} \coloneqq [2^{-i} n, 2^{-i+1}n] \cap \Z$.
   We find
   \begin{align*}
      \P\Bigl(\sup_{n-j\in I_i^{(n)}} \abs[\Big]{\frac{Y_j^{(n)} (n - j)}{\sqrt{n}}} \geq C-1\Bigr)
      &\leq \P\bigl(\sup_{n-j\in I_i^{(n)}} \abs{Y_j^{(n)} 2^{-i+1} \sqrt{n}} \geq C-1\bigr)\\
      &= \P\Bigl(\sup_{n-j\in I_i^{(n)}} \abs{Y_j^{(n)}} \geq \frac{2^{i-1} (C-1)}{\sqrt{n}}\Bigr)\\
      &\leq \frac{n}{2^{2i-2} (C-1)^2} \V(Y_{n - \ceil{2^{-i} n}}^{(n)}) \\
      &\leq \frac{n}{2^{2i-2} (C-1)^2} \cdot \frac{2^i}{n} = \frac{4}{2^i (C-1)^2},
   \end{align*}
   where in the last inequality we bounded the variance as follows:
   \[
      \V(Y_{n - \ceil{2^{-i} n}}^{(n)}) = \frac{\V(K_{n - \ceil{2^{-i} n}}^{(n)})}{\ceil{2^{-i} n}^2} = \frac{(n - \ceil{2^{-i} n})(n - \ceil{2^{-i} n} - 1) \ceil{2^{-i} n}}{n^2 \ceil{2^{-i} n}^2} \leq \frac{2^i}{n}.
   \]
   Finally, the union bound together with the observation that $\sum_{i\geq 1} \frac{4}{2^i (C-1)^2} = 4 (C-1)^{-2}$
   yields the upper bound in~\eqref{eq:edge-process:tight} and therefore completes
   the proof.
\end{proof}

Let us next consider the behavior of the finite-dimensional distributions.

\begin{lemma}\label{lem:edge-process:joint}
   Let $r$ be a fixed positive integer, and let $\mathbf{t} = (t_1, ..., t_r)\in (0,1)^r$.
   Then for $n\to\infty$, the random vector
   \[
      \mathbf{K}_{\floor{\mathbf{t} n}}^{(n)} \coloneqq (K_{\floor{t_1 n}}^{(n)}, K_{\floor{t_2 n}}^{(n)}, \dots, K_{\floor{t_r n}}^{(n)})
   \]
   converges, after centering and rescaling, for $n\to\infty$ in distribution
   to a multivariate normal distribution,
   \[
      \frac{\mathbf{K}_{\floor{\mathbf{t} n}}^{(n)} - \E \mathbf{K}_{\floor{\mathbf{t} n}}^{(n)}}{\sqrt{n}} \xrightarrow[n\to\infty]{d} \mathcal{N}(\mathbf{0}, \Sigma),
   \]
   where the expectation vector $\E \mathbf{K}_{\floor{\mathbf{t} n}}^{(n)}$
   satisfies
   \begin{equation}\label{eq:joint:exp}
      \E \mathbf{K}_{\floor{\mathbf{t} n}}^{(n)} = n(t_1^2, t_2^2, \dots, t_r^2) + \Oh(1),
   \end{equation}
   and the entries of the variance-covariance matrix $\Sigma = (\sigma_{i,j})_{1\leq i,j\leq r}$
   are
   \begin{equation}\label{eq:joint:var}
      \sigma_{i,j} = \begin{cases}
         t_i^2 (1 - t_j) & \text{ if } i \leq j,\\
         t_j^2 (1 - t_i) & \text{ if } i > j.
      \end{cases}
   \end{equation}
\end{lemma}
\begin{proof}
   As a consequence of Lemma~\ref{lemma:edges-gf} and Cayley's well-known
   enumeration formula for labeled trees of size $n$, we find that the probability
   generating function of the number of edge increments after
   $1 < j_1 < j_2 < \dots < j_r < n$ steps, respectively, is given by
   \begin{align}
      P_n(z_1, z_2, \dots, z_r) &= \frac{E_n(z_1, z_2, \dots, z_r)}{n^{n-2}} \notag\\
      &= \prod_{i=1}^r \Bigl(1 - \frac{j_r}{n} + \frac{j_i}{n}z_i
      + \sum_{h=i+1}^r \frac{j_h - j_{h-1}}{n} z_h\Bigr)^{j_i - j_{i-1}}, \label{eq:edges-pgf}
   \end{align}
   where $j_0 = 1$ for the sake of convenience. Given that this probability generating
   function factors nicely, we could use a general result like the
   multidimensional quasi-power theorem (cf.~\cite{Heuberger-Kropf:2018:quasipower-multidim})
   to prove that the corresponding random vector
   \[
      \Delta_{\mathbf{j}}^{(n)} = (K_{j_1}^{(n)}, K_{j_2}^{(n)} - K_{j_1}^{(n)}, \dots,
      K_{j_r}^{(n)} - K_{j_{r-1}}^{(n)})
   \]
   converges, after suitable rescaling, to a multivariate Gaussian limiting
   distribution. However, there is a simple probabilistic argument:
   Observe that $\Delta_{\mathbf{j}}^{(n)}$
   can be seen as a marginal distribution of the sum of $r$ independent,
   multinomially distributed random vectors:
   write $t_i = j_i/n$ and
   consider $M_j \sim \operatorname{Multi}(j_i - j_{i-1}, \mathbf{p}_i)$ where
   \begin{equation}\label{eq:multinomial-prob}
   \mathbf{p}_i = (p_{i,0}, p_{i,1}, \dots, p_{i, r}) \in [0,1]^r \quad \text{such that} \quad
   p_{i,h} = \begin{cases}
      1 - t_r & \text{ if } h = 0,\\
      0 & \text{ if } 0 < h < i,\\
      t_i & \text{ if } h = i,\\
      t_{h} - t_{h-1} & \text{ otherwise.}
      \end{cases}
   \end{equation}
   By construction, the probability generating function of $M_i$ is then given
   by
   \[ \Bigl((1 - t_r)z_0 + t_i z_i + \sum_{h=i+1}^{r}
   (t_h - t_{h-1}) \Bigr)^{j_i - j_{i-1}}, \]
   so that the probability generating function of the sum
   $M_1 + \dots + M_r$ is a product that is very similar (and actually equal
   if we set $z_0 = 1$, which corresponds to marginalizing out the first component)
   to~\eqref{eq:edges-pgf}. In order to make the following arguments formally easier
   to read, and as the first component is not relevant for us at all,
   we slightly abuse notation and let $M_i$ for $1\leq i\leq r$ denote the
   corresponding marginalized multinomial distributions instead.

   For the sake of convenience, we make a slight
   adjustment: instead of fixing the integer vector $\mathbf{j} = (j_1, \dots, j_r)$,
   we fix $\mathbf{t} = (t_1, \dots, t_r)$ with $0 < t_1 < \dots < t_r < 1$
   and define $\mathbf{j} = \floor{\mathbf{t} n}$. Here, $n$ is considered to be
   sufficiently large so that the conditions for the corresponding integer vector,
   $1 < \floor{t_1 n} < \dots < \floor{t_r n} < n$, are still satisfied.

   By the multivariate central limit theorem, it is well-known that
   a multinomially distributed random vector $M\sim \operatorname{Multi}(n, \mathbf{p})$
   converges, for $n\to\infty$ and after appropriate scaling, in distribution to
   a multivariate normal distribution,
   \begin{equation}\label{eq:multinomial-gaussian}
      \frac{M - n\mathbf{p}}{\sqrt{n}} \xrightarrow{d} \mathcal{N}(\mathbf{0}, \diag(\mathbf{p}) - \mathbf{p}^{\top} \mathbf{p}).
   \end{equation}
   As a consequence, we find that
   \begin{align*}
      \frac{\Delta_{\floor{n\mathbf{t}}}^{(n)} - \E \Delta_{\floor{n\mathbf{t}}}^{(n)}}{\sqrt{n}}
      &= \frac{(M_1 + \dots + M_r) - \E(M_1 + \dots + M_r)}{\sqrt{n}}   \\
      &= (\sqrt{t_1} + \Oh(n^{-1})) \frac{M_1 - \E M_1}{\sqrt{\floor{t_1 n}}} + \cdots \\
      &\qquad + (\sqrt{t_r - t_{r-1}} + \Oh(n^{-1}))\frac{M_r - \E M_r}{\sqrt{\floor{t_r n} - \floor{t_{r-1} n}}}\\
      & \xrightarrow[n\to\infty]{d} \sqrt{t_1}\mathcal{N}(\mathbf{0}, \Sigma_1) + \dots
      + \sqrt{t_r - t_{r-1}} \mathcal{N}(\mathbf{0}, \Sigma_r)\\
      &= \mathcal{N}(\mathbf{0}, t_1\Sigma_1 + \dots + (t_r - t_{r-1})\Sigma_r),
   \end{align*}
   where the variance-covariance matrices are given by
   \[
      \Sigma_j = \diag(\mathbf{p}_j) - \mathbf{p}_j^{\top} \mathbf{p}_j.
   \]
   By a straightforward (linear) transformation consisting of taking partial sums,
   the random vector of increments $\Delta_{\floor{\mathbf{t}n}}^{(n)}$
   can be transformed into $\mathbf{K}_{\floor{\mathbf{t}n}}^{(n)}$. This
   proves that $\mathbf{K}_{\floor{\mathbf{t}n}}^{(n)}$ converges, after
   centering and rescaling, to a multivariate normal distribution.

   The entries of the corresponding variance-covariance matrix can either be
   determined mechanically from the entries of $t_1 \Sigma_1 + \dots + (t_r - t_{r-1}) \Sigma_r$
   by taking the partial summation into account, or alternatively,
   our observations concerning the martingale $Y_j^{(n)}$ can be used.
   In particular, using~\eqref{eq:rescaled-martingale}, we find, for fixed $s,t\in [0,1]$
   with $s < t$, that
   \begin{align*}
      \Cov\Bigl(\frac{K_{\floor{sn}}^{(n)} - \E K_{\floor{sn}}^{(n)}}{\sqrt{n}}, \frac{K_{\floor{tn}}^{(n)} - \E K_{\floor{tn}}^{(n)}}{\sqrt{n}}\Bigr)
         &= \frac{(n - \floor{tn})(n - \floor{sn})}{n} \E(Y_{\floor{sn}}^{(n)} Y_{\floor{tn}}^{(n)})\\
         &= (n(1-t)(1-s) + \Oh(1))\E({Y_{\floor{sn}}^{(n)}}^2)\\
         &= s^2 (1 - t) + \Oh(n^{-1}),
   \end{align*}
   where we made use of the martingale property, and the fact that the second moment
   of $Y_j^{(n)}$ is equal to the variance $n^{-2} j (j-1) / (n-j)$.
   Ultimately, this verifies~\eqref{eq:joint:var} and thus completes the proof.
\end{proof}

The last remaining piece required to prove that the sequence of processes
$(Z^{(n)})_{n\geq 1}$ is tight is a bound on the growth of the corresponding
paths.

\begin{lemma}\label{lem:edge-process:paths-ok}
   There is a constant $\lambda$ such that the following inequality holds for all $s, t\in [0,1]$ and all $n\in\mathbb{Z}_{> 0}$:
   \begin{equation}\label{eq:edge-process:paths-ok}
      \E \bigl(\bigl|Z^{(n)}(t) - Z^{(n)}(s)\bigr|^4 \bigr) \leq \lambda\, \abs{t-s}^2.
   \end{equation}
\end{lemma}
\begin{proof}
Let us first consider the case where both $sn = \ell$ and $tn = m$ are integers. Assume that $\ell > m$. We can follow the argument in the proof of Lemma~\ref{lem:edge-process:joint} to see that the random variable $K_{\ell}^{(n)} - K_{m}^{(n)}$ is distributed like the last component in the sum of two independent multinomial distributions, which gives us
\[ K_{\ell}^{(n)} - K_{m}^{(n)} \sim \operatorname{Bin} \Big(m - 1, \frac{\ell-m}{n} \Big) + \operatorname{Bin} \Big(\ell-m, \frac{\ell}{n} \Big).\]
Let us write $\operatorname{Bin}^*(n,p)$ for a centered binomial distribution, i.e., a binomial distribution $\operatorname{Bin}(n,p)$ with the mean $np$ subtracted. Then we have
\begin{equation}\label{eq:KlKm}
K_{\ell}^{(n)} - K_{m}^{(n)} - \frac{\ell^2-m^2}n \sim \operatorname{Bin}^* \Big(m - 1, \frac{\ell-m}{n} \Big) + \operatorname{Bin}^* \Big(\ell-m, \frac{\ell}{n} \Big) - \frac{\ell-m}{n}.
\end{equation}
The fourth moment of a $\operatorname{Bin}^*(n,p)$-distributed random variable, which is the fourth centered moment of a $\operatorname{Bin}(n,p)$-distributed random variable, is
\[np(1-p)\big(1+(3n-6)p(1-p)\big) \leq np(1+3np).\]
Note that for the $\operatorname{Bin}^*$-variables in~\eqref{eq:KlKm}, we have $\frac{(m-1)(\ell-m)}{n} \leq \ell-m$ and $\frac{(\ell-m)\ell}{n} \leq \ell-m$. Thus the fourth moments of the two centered binomial random variables are bounded above by
\[(\ell-m)(1+3(\ell-m)) \leq 4(\ell-m)^2.\]
So $K_{\ell}^{(n)} - K_{m}^{(n)} - \frac{\ell^2-m^2}n$ is the sum of three random variables (the third one actually constant) whose fourth moments are bounded above by $4(\ell-m)^2$, $4(\ell-m)^2$ and $(\ell-m)^4n^{-4} \leq (\ell-m)^2$ respectively. Applying the inequality $\E((X+Y+Z)^4) \leq 27(\E(X^4) + \E(Y^4) + \E(Z^4))$, which is a simple consequence of Jensen's inequality, we get
\[\E \Big( K_{\ell}^{(n)} - K_{m}^{(n)} - \frac{\ell^2-m^2}n \Big)^4 \leq 27 \cdot (4+4+1)(\ell-m)^2 = 243(\ell-m)^2.\]
So if $tn = \ell$ and $sn = m$ are integers, we have
\begin{align*}
\E \bigl(\bigl|Z^{(n)}(t) - Z^{(n)}(s)\bigr|^4 \bigr) &= \E \Big( \frac{K_{\ell}^{(n)} - \frac{\ell^2}{n}}{\sqrt{n}} - \frac{K_{m}^{(n)} - \frac{m^2}{n}}{\sqrt{n}} \Big)^4 \\
&\leq \frac{243(\ell-m)^2}{n^2} = 243(t-s)^2.
\end{align*}
Second, consider the case that $tn$ and $sn$ lie between two consecutive integers $m$ and $m+1$: $\tilde s n = m \leq sn \leq tn \leq m+1 = \tilde t n$. In this case, we can express the difference $Z^{(n)}(t) - Z^{(n)}(s)$ (using the linear interpolation in the definition of $\tilde K_t^{(n)}$) as
\[Z^{(n)}(t) - Z^{(n)}(s) = (t-s)n \bigl( Z^{(n)}(\tilde t) - Z^{(n)}(\tilde s) \bigr) + (t-s) \sqrt{n} (\tilde t - t - s + \tilde s).\]
The fourth moment of the first term is
\[\E \Big( (t-s)n \bigl( Z^{(n)}(\tilde t) - Z^{(n)}(\tilde s) \bigr) \Big)^4 \leq (t-s)^4n^4 \cdot 243 (\tilde t - \tilde s)^2 \leq 243(t-s)^2\]
since $|t-s| \leq |\tilde t - \tilde s| = \frac1n$. Likewise, the fourth power of the second term is easily seen to be bounded above by $(t-s)^2$. So the elementary inequality $\E((X+Y)^4) \leq 8(\E(X^4) + \E(Y^4))$ yields
\[E \bigl(\bigl|Z^{(n)}(t) - Z^{(n)}(s)\bigr|^4 \bigr) \leq 8(243+1)(t-s)^2 = 1952(t-s)^2.\]
Finally, in the general case that $t$ and $s$ are arbitrary real numbers in the interval $[0,1]$ such that $tn$ and $sn$ do not lie between consecutive integers, we can write
\[Z^{(n)}(t) - Z^{(n)}(s) = \bigl( Z^{(n)}(t) - Z^{(n)}(u_1) \bigr) + \bigl( Z^{(n)}(u_1) - Z^{(n)}(u_2) \bigr) + \bigl( Z^{(n)}(u_2) - Z^{(n)}(s) \bigr)\]
for some real numbers $u_1,u_2$ with $s \leq u_2 \leq u_1 \leq t$ such that $u_1n$ and $u_2n$ are integers and $tn \leq u_1n+1$ as well as $sn \geq u_2n-1$. Combining the bounds from above and using again the inequality $\E((X+Y+Z)^4) \leq 27(\E(X^4) + \E(Y^4) + \E(Z^4))$, we obtain
\begin{align*}
\E \bigl(\bigl|Z^{(n)}(t) - Z^{(n)}(s)\bigr|^4 \bigr) &\leq 27 \bigl( 1952(t-u_1)^2 + 243(u_1-u_2)^2 + 1952(u_2-s)^2 \bigr) \\
&\leq 27 \cdot 1952 (t-u_1+u_1-u_2+u_2-s)^2 = 52704(t-s)^2,
\end{align*}
completing the proof of the lemma with $\lambda = 52704$.
\end{proof}

All that remains now is to combine the two ingredients to prove
our main result on the limiting process.

\begin{proof}[Proof of Theorem~\ref{thm:edge-process}]
   The proof relies on the well-known result
   asserting that given tightness of the sequence of corresponding probability
   measures as well as convergence of the finite-dimensional probability distributions,
   a sequence of stochastic processes converges to a limiting process
   (see~\cite[Theorem 7.1, Theorem 7.5]{Billingsley:1999:covergence-probability}).

   Tightness is implied (see \cite[Theorems 13.29, 21.42]{Klenke:2020:probability-theory-course})
   by tightness of the initial distributions via Lemma~\ref{lem:edge-process:tight}
   and the moment bound in Lemma~\ref{lem:edge-process:paths-ok}.
   The (limiting) behavior
   of the finite-dimensional distributions of the original process
   $(K_{\floor{tn}}^{(n)})_{t\in[0,1]}$ is characterized by
   Lemma~\ref{lem:edge-process:joint}. This characterization carries over
   to the linearly interpolated process by an application of Slutsky's
   theorem \cite[Theorem 13.18]{Klenke:2020:probability-theory-course} after
   observing that
   \begin{align*}
      \P\biggl(\abs[\bigg]{Z^{(n)}(t) - \frac{K_{\floor{tn}}^{(n)} - t^2 n}{\sqrt{n}}} > \varepsilon\biggr)
      & = \P\Bigl(\frac{\abs{\tilde K_{t}^{(n)} - K_{\floor{tn}}^{(n)}}}{\sqrt{n}} > \varepsilon\Bigr)
      \leq \frac{ \E((\tilde K_{t}^{(n)} - K_{\floor{tn}}^{(n)})^2) }{n\varepsilon^2}\\
      &\leq \frac{ \E( (K_{\floor{tn}+1}^{(n)} - K_{\floor{tn}}^{(n)})^2 )}{n \varepsilon^{2}}
      \xrightarrow{n\to\infty} 0,
   \end{align*}
   as a mechanical computation shows that
   $\E( (K_{\floor{tn}+1}^{(n)} - K_{\floor{tn}}^{(n)})^2 ) = \Oh(1)$ (this also follows from Lemma~\ref{lem:edge-process:paths-ok}).

   Note that as the finite-dimensional distributions converge to Gaussian
   distributions, the limiting process $(Z^{\infty}(t))_{t\in[0,1]}$
   is Gaussian itself---which means that it is fully characterized by its
   first and second order moments. As a consequence of Lemma~\ref{lem:edge-process:joint},
   we find for all $s, t\in [0,1]$ with $s < t$ that
   \[
         \E Z^{\infty}(t) = 0,\qquad
         \V Z^{\infty}(t) = t^2 (1 - t), \qquad
         \Cov(Z^{\infty}(s), Z^{\infty}(t)) = s^2 (1-t).
   \]
   It can be checked that if $(W(t))_{t\in[0,1]}$ is a standard Wiener process,
   the Gaussian process $((1-t)W(t^2/(1-t)))_{t\in[0,1]}$ has the same first
   and second order moments and therefore also the same distribution as $Z^{\infty}$.
\end{proof}

\section{Size of the root cluster}\label{sec:root-cluster}

We now shift our attention from the number of uncovered edges to the sizes of the
connected components (or \emph{clusters}) appearing in the graph throughout the uncover process.
It will prove convenient to change our tree model to \emph{rooted} labeled trees, as the
nature of rooted trees allows us to focus our investigation on one particular cluster -- the one
containing the root vertex. In case the root vertex has not yet been uncovered, we will consider
the size of the root cluster to be 0. Formally, we let the random variable $R_{n}^{(k)}$
be the size of the root cluster of a (uniformly) random rooted labeled tree of size
$n$ with $k$ uncovered vertices.

Using the symbolic method for labeled structures
(cf.~\cite[Chapter II]{Flajolet-Sedgewick:2009:analy}), we can set up a formal specification
for the corresponding combinatorial classes and subsequently extract functional equations
for the associated generating functions. Let $\Troot$ be the class of rooted labeled trees, and
let $\mathcal{G}$ be a refinement of $\Troot$ where the vertices can either be
covered or uncovered, and where uncovered vertices are marked with a marker $U$. Finally, let
$\mathcal{F}$ be a further refinement of $\mathcal{G}$ where all uncovered vertices in
the root cluster are additionally marked with marker $V$. A straightforward ``top-down approach'',
i.e., a decomposition of the members of the tree family w.r.t.\ the root vertex, yields the formal
specification
\begin{equation*}
  \mathcal{F} = \mathcal{Z} \ast \textsc{Set}(\mathcal{G}) + \mathcal{Z} \times \{U,V\} \ast \textsc{Set}(\mathcal{F}), \qquad
	\mathcal{G} = \mathcal{Z} \ast \textsc{Set}(\mathcal{G}) + \mathcal{Z} \times \{U\} \ast \textsc{Set}(\mathcal{G}).
\end{equation*}
Note that the first summand in the formal description of $\mathcal{F}$ corresponds to the case where the root vertex is covered, thus the size of the root cluster is zero.

Introducing the corresponding exponential generating functions $F:=F(z,u,v)$ and $G:=G(z,u)$,
\begin{small}
\begin{align*}
  F(z,u,v) & := \sum_{T \in \mathcal{F}} \frac{z^{|T|} u^{\text{$\# U$ in $T$}} v^{\text{$\# V$ in $T$}}}{|T|!} = \sum_{n \ge 1} \sum_{0 \le k \le n} \sum_{m \ge 0} \frac{n^{n-1}}{n!} \binom{n}{k} \mathbb{P}\{R_{n}^{(k)} = m\} z^{n} u^{k} v^{m},\\
	G(z,u) & := \sum_{T \in \mathcal{G}} \frac{z^{|T|} u^{\text{$\# U$ in $T$}}}{|T|!} = \sum_{n \ge 1} \sum_{0 \le k \le n} \frac{n^{n-1}}{n!} \binom{n}{k} z^{n} u^{k},
\end{align*}
\end{small}
we obtain the characterizing equations
\begin{equation}\label{eqn:FG_R_labeled}
  F = z e^{G} + zuv e^{F}, \qquad G = z (1+u) e^{G}.
\end{equation}
Of course, $G(z,u) = T^{\bullet}(z(1+u))$, where $T^{\bullet}$ is the exponential generating
function associated with $\Troot$, the Cayley tree function.
Starting with \eqref{eqn:FG_R_labeled}, the following results on $R_{n}^{(k)}$ can be deduced.

\begin{theorem}\label{thm:R_Exp_labeled}
The expectation $\mathbb{E}(R_{n}^{(k)})$ is, for $0 \le k \le n$ and $n \ge 1$, given by
\begin{equation}\label{eq:R_Exp_labeled:explicit-sum}
  \mathbb{E}(R_{n}^{(k)}) = \sum_{j=1}^{k} \frac{j \, k^{\underline{j}}}{n^{j}}.
\end{equation}

Depending on the growth of $k=k(n)$, $\mathbb{E}(R_{n}^{(k)})$ has the following asymptotic behavior:
\begin{equation*}
  \mathbb{E}(R_{n}^{(k)}) \sim
	\begin{cases}
	  \frac{k}{n}, & \text{for $k=o(n)$}, \quad \text{($k$ small)},\\
		\frac{\alpha}{(1-\alpha)^{2}}, & \text{for $k \sim \alpha n$, with $0 < \alpha < 1$}, \quad \text{($k$ in central region)},\\
		\frac{n^{2}}{d^{2}}, & \text{for $k=n-d$, with $d = \omega(\sqrt{n})$ and $d = o(n)$},\\[-0.5ex]
		& \quad \text{($k$ subcritically large)},\\
		\kappa n, & \mspace{-54mu}\text{with} \quad \kappa = 1-c e^{\frac{c^{2}}{2}} \int_{c}^{\infty} e^{-\frac{t^{2}}{2}} dt,\\[-0.5ex]
		& \text{for $k=n-d$, with $d \sim c \sqrt{n}$ and $c > 0$},\\[-0.5ex]
		& \quad \text{($k$ critically large)},\\
		n - \sqrt{\frac{\pi}{2}} d \sqrt{n}, & \text{for $k=n-d$, with $d = o(\sqrt{n})$},\\[-0.5ex]
		& \quad \text{($k$ supercritically large)}.
		\end{cases}
\end{equation*}

\end{theorem}
\begin{proof}
   After introducing $E:= E(z,u) = \left.\frac{\partial}{\partial v}F(z,u,v)\right|_{v=1} = \sum_{n,k} \frac{n^{n-1}}{n!} \binom{n}{k} \mathbb{E}(R_{n}^{(k)})$,
   considering the partial derivative of~\eqref{eqn:FG_R_labeled} with respect to $v$ easily yields
   \begin{equation*}
      E = \frac{1}{1-\frac{u}{1+u}G} - 1.
   \end{equation*}

   Extracting coefficients of $E$ by an application of the Lagrange inversion formula (see, e.g., \cite[Theorem A.2]{Flajolet-Sedgewick:2009:analy}) yields
	\begin{align*}
	  [z^{n}] E & = \frac{1}{n} [G^{n-1}] \frac{u}{(1+u) (1-\frac{u}{1+u} G)^{2}} \cdot (1+u)^{n} e^{n G}\\
		& = \sum_{j=0}^{n-1} (j+1) u^{j+1} (1+u)^{n-1-j} \cdot \frac{n^{n-2-j}}{(n-1-j)!},
	\end{align*}
	and further
	\begin{equation*}
	  [z^{n} u^{k}] E = \sum_{j=0}^{k-1} (j+1) \binom{n-1-j}{k-1-j} \frac{n^{n-2-j}}{(n-1-j)!}.
	\end{equation*}
  Using $\mathbb{E}(R_{n}^{(k)}) = \frac{[z^{n} u^{k}] E}{[z^{n} u^{k}] G} = \frac{n! \, [z^{n} u^{k}] E}{n^{n-1} \binom{n}{k}}$, we obtain the result stated in \eqref{eq:R_Exp_labeled:explicit-sum}:
	\begin{align*}
	  \mathbb{E}(R_{n}^{(k)}) & = n! \sum_{j=0}^{k-1} (j+1) \frac{\binom{n-1-j}{k-1-j}}{\binom{n}{k}} \cdot \frac{n^{-1-j}}{(n-1-j)!}
		= n! \sum_{j=0}^{k-1} \frac{j+1}{n^{1+j} (n-1-j)!} \cdot \frac{k^{\underline{j+1}}}{n^{\underline{j+1}}}\\
		& = \sum_{j=1}^{k} \frac{j \, k^{\underline{j}}}{n^{j}}.
	\end{align*}

	In order to analyze the asymptotic behavior of $\mathbb{E}(R_{n}^{(k)})$, the following integral representation turns out to be advantageous,
   \begin{equation}\label{eq:R_Exp_labeled:integral_representation}
   \mathbb{E}(R_{n}^{(k)}) = \int_{0}^{\infty} (x-1) \, e^{-x} \Big(1+\frac{x}{n}\Big)^{k} dx.
   \end{equation}
	It can be verified in a straightforward way by using the integral representation of the Gamma function:
	\begin{align*}
	  & \int_{0}^{\infty} (x-1) \, e^{-x} \Big(1+\frac{x}{n}\Big)^{k} dx = \int_{0}^{\infty} (x-1) e^{-x} \sum_{j=0}^{k} \binom{k}{j} \frac{x^{j}}{n^{j}} dx\\
		& \quad = \sum_{j=0}^{k} \frac{\binom{k}{j}}{n^{j}} \left(\int_{0}^{\infty} e^{-x} x^{j+1} dx - \int_{0}^{\infty} e^{-x} x^{j} dx\right)
		= \sum_{j=0}^{k} \frac{\binom{k}{j}}{n^{j}} \big((j+1)! - j!\big)\\
		& \quad = \sum_{j=0}^{k} \frac{k^{\underline{j}} \, j \, j!}{j! \, n^{j}} = \sum_{j=0}^{k} \frac{j \, k^{\underline{j}}}{n^{j}}.
	\end{align*}

	In order to evaluate the integral~\eqref{eq:R_Exp_labeled:integral_representation} asymptotically, we first show that for the range $x \ge n^{\frac{1}{2} + \epsilon}$, with arbitrary but fixed $\epsilon > 0$, the contribution to the integral is exponentially small and thus negligible. Namely, when considering the integrand and setting $t = \frac{x}{n}$, we obtain the following estimate, uniformly for $k \in [0,n]$:
	\begin{equation*}
	  e^{-x} \Big(1+\frac{x}{n}\Big)^{k} \le e^{-x} \Big(1+\frac{x}{n}\Big)^{n} = e^{-x + n \ln(1+\frac{x}{n})} = e^{-n \big(t-\ln(1+t)\big)}.
	\end{equation*}
	Simple monotonicity considerations yield $t - \ln(1+t) \ge \frac{t}{4}$, for $t \ge 1$ (thus $x \ge n$), and $t - \ln(1+t) \ge \frac{t^{2}}{4}$, for $t \in [0,1]$ (thus $x \in [0,n]$). Due to these estimates and by evaluating the resulting integral, we get the following bounds on the integral for the respective ranges:
	\begin{align*}
	  \int_{n}^{\infty} (x-1) e^{-x} \Big(1+\frac{x}{n}\Big)^{k} dx & \le \int_{n}^{\infty} x e^{-\frac{x}{4}} dx = 16 \big(1+\frac{n}{4}\big) e^{-\frac{n}{4}},\\
		\int_{n^{\frac{1}{2} + \epsilon}}^{n} (x-1) e^{-x} \Big(1+\frac{x}{n}\Big)^{k} dx & \le \int_{n^{\frac{1}{2} + \epsilon}}^{\infty} x e^{-\frac{x^{2}}{4n}} dx = 2n e^{-\frac{1}{4} n^{2 \epsilon}},
	\end{align*}
	thus, by combining both cases, uniformly for $k \in [0,n]$ and $\epsilon \in (0, \frac{1}{2}]$:
	\begin{equation*}
	  \int_{n^{\frac{1}{2} + \epsilon}}^{\infty} (x-1) e^{-x} \Big(1+\frac{x}{n}\Big)^{k} dx = \Oh\big(n e^{-\frac{1}{4} n^{2 \epsilon}}\big).
	\end{equation*}

	Furthermore, we note that (roughly speaking) if $k$ is sufficiently far away from $n$ the integration range with negligible contribution can be extended. Namely, setting $\delta = 1-\frac{k}{n}$ and $x=nt$, we obtain for the integrand
	\begin{equation*}
	  e^{-x} \Big(1+\frac{x}{n}\Big)^{k} = e^{-n \big(t-\frac{k}{n}\ln(1+t)\big)} = e^{-n \big(t-(1-\delta)\ln(1+t)\big)} \le e^{-n \delta t} = e^{-\delta x} = e^{-(1-\frac{k}{n})x},
	\end{equation*}
	where we used the trivial estimate $\ln(1+t) \le t$, for $t \ge 0$. E.g., if $\delta \ge n^{-\frac{1}{4}}$, i.e., $k \le n - n^{\frac{3}{4}}$, one easily obtains from this estimate that the contribution to the integral from the range $x \ge n^{\frac{1}{4} + \epsilon}$ is asymptotically negligible:
	\begin{equation*}
	  \int_{n^{\frac{1}{4} + \epsilon}}^{n^{\frac{1}{2} + \epsilon}} (x-1) e^{-x} \Big(1+\frac{x}{n}\Big)^{k} dx \le \int_{n^{\frac{1}{4} + \epsilon}}^{\infty} x e^{-(1-\frac{k}{n})x} dx \le \int_{n^{\frac{1}{4} + \epsilon}}^{\infty} x e^{-n^{-\frac{1}{4}}x} dx = \Oh\big(n e^{-n^{\epsilon}}\big).
	\end{equation*}

	To get the asymptotic expressions for the integral stated in the theorem, we will take into account the growth of $k$ w.r.t.\ $n$, consider suitable expansions of the integrand for the range $x \le n^{\frac{1}{2}+\epsilon}$ (or $x \le n^{\frac{1}{4}+\epsilon}$, respectively) and evaluate the resulting integrals, where we apply the ``tail exchange technique'', i.e., we may extend the integration range to $x \ge 0$, since only asymptotically negligible contributions are added.
	\begin{itemize}
	  \item $k$ small or in the central region: assuming $k \le n - n^{\frac{3}{4}}$, an expansion of the integrand for $x \le n^{\frac{1}{4} + \epsilon}$ leads to the expansion (with a uniform bound in this range):
		\begin{align*}
		  e^{-x} \Big(1+\frac{x}{n}\Big)^{k} & = e^{-x + k \ln(1+\frac{x}{n})} = e^{-x (1-\frac{k}{n}) + \Oh(\frac{k x^{2}}{n^{2}})}
			= e^{-x(1-\frac{k}{n})} \cdot \Big(1+\Oh\big({\textstyle{\frac{k x^{2}}{n^{2}}}}\big)\Big)\\
			& = e^{-x(1-\frac{k}{n})} \cdot \Big(1+\Oh\big(n^{-\frac{1}{2}+2\epsilon}\big)\Big),
		\end{align*}
		and thus to the following evaluation of the integral:
		\begin{align*}
		  & \int_{0}^{n^{\frac{1}{4}+\epsilon}} (x-1) e^{-x} \Big(1+\frac{x}{n}\Big)^{k} dx = \int_{0}^{n^{\frac{1}{4}+\epsilon}} (x-1) e^{-x(1-\frac{k}{n})} dx \cdot \Big(1+\Oh\big(n^{-\frac{1}{2}+2\epsilon}\big)\Big) \\
			& \quad = \int_{0}^{\infty} (x-1) e^{-x(1-\frac{k}{n})} dx \cdot \Big(1+\Oh\big(n^{-\frac{1}{2}+2\epsilon}\big)\Big) = \frac{\frac{k}{n}}{(1-\frac{k}{n})^{2}} \cdot \Big(1+\Oh\big(n^{-\frac{1}{2}+2\epsilon}\big)\Big),
		\end{align*}
		where we used the formula
		\begin{equation}\label{eq:integral_evaluation}
		  \int_{0}^{\infty} (x-1) e^{-\nu x} dx = \frac{1-\nu}{\nu^{2}}, \quad \text{for $\nu > 0$}.
		\end{equation}
		Of course, this gives in particular
		\begin{equation*}
		  \mathbb{E}(R_{n}^{(k)}) \sim
			\begin{cases} \frac{k}{n}, & \quad \text{for $k = o(n)$},\\
			\frac{\alpha}{(1-\alpha)^{2}}, & \quad \text{for $k \sim \alpha n$, with $0 <\alpha < 1$}.
			\end{cases}
		\end{equation*}
		\item $k$ subcritically large: in the following we treat cases with $\frac{k}{n} \to 1$; there we have to distinguish according to the growth behaviour of the difference $d=n-k$. First we examine the region $\sqrt{n} \ll d \ll n$, i.e., $d = o(n)$, but $d = \omega(\sqrt{n})$, for which we get the following expansion of the integrand (for $x \le n^{\frac{1}{2} + \epsilon}$):
		\begin{align*}
		  e^{-x} \Big(1+\frac{x}{n}\Big)^{k} & = e^{-x + (n-d) \ln(1-\frac{x}{n})} = e^{-\frac{d}{n} x + \Oh(\frac{x^{2}}{n}) + \Oh(\frac{x^{3}}{n^{2}})}\\
			& = e^{-\frac{dx}{n}} \cdot \Big(1+\Oh\big({\textstyle{\frac{x^{2}}{n}}}\big) + \Oh\big({\textstyle{\frac{x^{3}}{n^{2}}}}\big)\Big).
		\end{align*}
		Considering the corresponding integral (and applying tail exchange) we obtain
		\begin{multline*}
		  \mathbb{E}(R_{n}^{(k)}) = \int_{0}^{\infty} (x-1) e^{-\frac{d x}{n}} dx\\
			\mbox{} + \Oh\left(\frac{1}{n} \cdot \int_{0}^{\infty} (x-1) x^{2} e^{-\frac{d x}{n}} dx\right) + \Oh\left(\frac{1}{n^{2}} \cdot \int_{0}^{\infty} (x-1) x^{3} e^{-\frac{d x}{n}} dx\right).
		\end{multline*}
		Using \eqref{eq:integral_evaluation} and
		\begin{equation*}
		  \int_{0}^{\infty} (x-1) x^{\ell} e^{-\nu x} dx = \Oh\big({\textstyle{\frac{1}{\nu^{\ell+2}}}}\big), \quad \text{for $\ell \ge 0$ and $\nu \in (0,1)$},
		\end{equation*}
		we obtain the stated result:
		\begin{equation*}
		  \mathbb{E}(R_{n}^{(k)}) = \frac{n^{2}}{d^{2}} + \Oh\big({\textstyle{\frac{n}{d}}}\big) + \Oh\big({\textstyle{\frac{n^{3}}{d^{4}}}}\big)
			= \frac{n^{2}}{d^{2}} \cdot \Big(1+\Oh\big({\textstyle{\frac{d}{n}}}\big) + \Oh\big({\textstyle{\frac{n}{d^{2}}}}\big)\Big)
			\sim \frac{n^{2}}{d^{2}}.
		\end{equation*}
		\item $k$ critically large: for the case that the difference $d=n-k$ is of order $\Theta(\sqrt{n})$, we obtain the following expansion of the integrand:
		\begin{equation*}
		  e^{-x} \Big(1+\frac{x}{n}\Big)^{k} = e^{-\frac{d x}{n} - \frac{x^{2}}{2n}} \cdot \Big(1+\Oh\big({\textstyle{\frac{d x^{2}}{n^{2}}}}\big) + \Oh\big({\textstyle{\frac{x^{3}}{n^{2}}}}\big)\Big).
		\end{equation*}
		Thus, after completing the integrals occurring, we get
		\begin{align*}
		  \mathbb{E}\big(R_{n}^{(k)}\big) & = \int_{0}^{\infty} (x-1) e^{-\frac{d x}{n} - \frac{x^{2}}{2n}} \, dx\\
			& \qquad \mbox{} + \Oh\Big(\frac{d}{n^{2}} \cdot \int_{0}^{\infty} x^{3} e^{-\frac{d x}{n} - \frac{x^{2}}{2n}} \, dx\Big)
			+ \Oh\Big(\frac{1}{n^{2}} \cdot \int_{0}^{\infty} x^{4} e^{-\frac{d x}{n} - \frac{x^{2}}{2n}} \, dx\Big).
		\end{align*}
		Since, for $\ell \ge 0$,
		\begin{equation*}
		  \int_{0}^{\infty} x^{\ell} e^{-\frac{d x}{n} - \frac{x^{2}}{2n}} dx = \Oh\Big( \int_{0}^{\infty} x^{\ell} e^{- \frac{x^{2}}{2n}} dx \Big)
			= \Oh\big(n^{\frac{\ell+1}{2}}\big),
		\end{equation*}
		this yields
		\begin{equation*}
		  \mathbb{E}\big(R_{n}^{(k)}\big) = \int_{0}^{\infty} x \, e^{-\frac{d x}{n} - \frac{x^{2}}{2n}} dx + \Oh(\sqrt{n}).
		\end{equation*}
		Setting $d = c \sqrt{n}$ and applying the substitution $t=c+\frac{x}{\sqrt{n}}$, we evaluate the integral obtaining the stated result:
		\begin{equation*}
		  \int_{0}^{\infty} x \, e^{-\frac{d x}{n} - \frac{x^{2}}{2n}} \, dx = n \int_{c}^{\infty} (t-c) \, e^{\frac{c^{2}}{2} - \frac{t^{2}}{2}} dt
			= n \left(1-c e^{\frac{c^{2}}{2}} \cdot \int_{c}^{\infty} e^{-\frac{t^{2}}{2}} dt\right).
		\end{equation*}
		\item $k$ supercritically large: for $d=n-k = o(\sqrt{n})$, we get the expansion
		\begin{equation*}
		  e^{-x} \Big(1+\frac{x}{n}\Big)^{k} % = e^{-\frac{d x}{n} - \frac{x^{2}}{2n}} \cdot \Big(1+\Oh\big({\textstyle{\frac{d x^{2}}{n^{2}}}}\big) + \Oh\big({\textstyle{\frac{x^{3}}{n^{2}}}}\big)\Big)
			= e^{-\frac{x^{2}}{2n}} \cdot \big(1-\frac{dx}{n}\big) \cdot \Big(1+\Oh\big({\textstyle{\frac{d^{2} x^{2}}{n^{2}}}}\big) + \Oh\big({\textstyle{\frac{x^{3}}{n^{2}}}}\big)\Big).
		\end{equation*}
		Computations analogous to the previous ones, using
		\begin{equation*}
		  \int_{0}^{\infty} x^{\ell} e^{-\frac{x^{2}}{2n}} \, dx = 2^{\frac{\ell-1}{2}} \Gamma\big({\textstyle{\frac{\ell+1}{2}}}\big) \cdot n^{\frac{\ell+1}{2}}, \quad \text{for $\ell \ge 0$},
		\end{equation*}
		lead to the stated result:
		\begin{align*}
		  \mathbb{E}\big(R_{n}^{(k)}\big) & = \int_{0}^{\infty} x e^{-\frac{x^{2}}{2n}} dx
			- \frac{d}{n} \int_{0}^{\infty} x^{2} e^{-\frac{x^{2}}{2n}} dx + \Oh(d^{2}) + \Oh(\sqrt{n})\\
			& = n - \frac{\sqrt{\pi}}{\sqrt{2}} d \sqrt{n} + \Oh(d^{2}) + \Oh(\sqrt{n}).
		\end{align*}
	\end{itemize}
\end{proof}

We can even obtain the exact distribution of $R_{n}^{(k)}$. There are two different approaches
we want to briefly sketch: for one, an explicit formula for the generating function
$F=F(z,u,v)$ can be found either from manipulating the recursive
description~\eqref{eqn:FG_R_labeled}, or directly by decomposing $\mathcal{F}$
as a tree forming the uncovered root cluster with a forest with covered roots
attached. Either way, this yields
\begin{equation*}
  F = T^{\bullet}\big(v X e^{-X}\big) + \frac{G}{1+u}, \quad \text{with} \quad X= \frac{u G}{1+u}.
\end{equation*}
Note that the second summand, $\frac{G}{1+u} = z e^{G}$, corresponds to the case where the root vertex is covered.
Extracting coefficients via an application of the Lagrange inversion formula then yields an explicit formula for $F_{n,k,m} := n! [z^{n} u^{k} v^{m}] F(z,u,v)$, i.e., the number of labeled rooted trees with $n$ vertices of which $k$ are uncovered and $m$ belong to the root cluster (for $0 \le m \le k \le n$ and $n \ge 1$):
\begin{equation*}
  F_{n,k,m} =
	\begin{cases}
	  \binom{n-1}{k} n^{n-1}, & \quad m =0,\\
	  \binom{n}{m} \binom{n-m-1}{k-m} n^{n-k-1} m^{m} (n-m)^{k-m}, & \quad m \ge 1.
	\end{cases}
\end{equation*}
From this formula, the probabilities $\P(R_{n}^{(k)} = m) = \frac{F_{n,k,m}}{n^{n-1} \binom{n}{k}}$ given in Theorem~\ref{thm:R_Prob_labeled} can be obtained directly. We omit these straightforward, but somewhat lengthy computations, since in the following the results are deduced in a more general and elegant way.

Alternatively, there is also a more combinatorial approach to determine these
probabilities: there is an elementary formula enumerating trees where a specified
set of vertices forms a cluster.

\begin{lemma}\label{claim:tree-enumeration}
   Let $n$ and $k$ be positive integers, and let $r_1,r_2,\ldots,r_{\ell}$ be fixed positive integers with $r_1 +\cdots + r_{\ell} \leq k$. Moreover, fix disjoint subsets $R_1,\ldots,R_{\ell}$ of $[k]$ with $|R_i| = r_i$ for all $i$. The number of $n$-vertex labeled trees for which $R_1,\ldots,R_{\ell}$ are components of the forest induced by the vertices with labels in $[k]$ is given by
   \begin{equation}\label{eq:tree-enumeration:formula}
      n^{n-k-1}\big(n-r_1-\cdots-r_{\ell}\big)^{k-r_1-\cdots-r_{\ell}-1} (n-k)^{\ell} r_1^{r_1-1}\cdots r_{\ell}^{r_{\ell}-1}.
   \end{equation}
   \end{lemma}
   \begin{proof}
   We interpret each such tree $T$ as a spanning tree of a complete graph $K$ with $n$ vertices. Set $r = r_1 + \cdots +r_{\ell}$. The vertices of $K$ can be divided into the sets $R_1,\ldots,R_{\ell}$, the remaining $k - r$ vertices in $\{1,2,\ldots,k\}$ forming a set $Q$, and the $n-k$ vertices with label greater than $k$ forming a set $S$. Note first that the components induced by the sets $R_1,\ldots,R_{\ell}$ can be chosen in $r_1^{r_1-2} \cdots r_{\ell}^{r_\ell-2}$ ways. If the vertex sets corresponding to $R_1,\ldots,R_{\ell}$ are contracted to single vertices $v_1,\ldots,v_{\ell}$, $K$ becomes a multigraph $K'$ with $n - r + \ell$ vertices, and the tree $T$ becomes a spanning tree $T'$ of $K'$ upon contraction. Note that there are $r_i$ edges from $v_i$ to every other vertex in $K'$, and that $T'$ cannot contain any edges from $v_i$ to another vertex in $\{v_1,\ldots,v_{\ell} \} \cup Q$. Thus $T'$ remains a spanning tree if all such edges are removed from $K'$ to obtain a multigraph $K''$. Conversely, if we take an arbitrary spanning tree of $K''$ and replace the vertices $v_1,\ldots,v_{\ell}$ by spanning trees of  $R_1,\ldots,R_{\ell}$ respectively, we obtain a labeled tree with $n$ vertices that has the desired properties. It remains to count spanning trees of $K''$, which has an adjacency matrix of the block form
   $$A = \begin{bmatrix} O & O & \mathbf{r} \mathbf{1}^T \\ O & E - I & E \\ \mathbf{1} \mathbf{r}^T & E & E - I \end{bmatrix}$$
   Here, $\mathbf{1}$ denotes a (column) vector of $1$s, $\mathbf{r}$ a (column) vector whose entries are $r_1,\ldots,r_{\ell}$, $O$ a matrix of $0$s, $E$ a matrix of $1$s, and $I$ an identity matrix. The blocks correspond to $\ell$, $k-r$ and $n-k$ rows/columns, respectively. The number of spanning trees can now be determined by means of the matrix-tree theorem: the Laplacian matrix is given by
   $$L = \begin{bmatrix} (n-k) D & O & - \mathbf{r} \mathbf{1}^T \\ O & (n-r) I - E & -E \\ -\mathbf{1} \mathbf{r}^T & -E & n I - E \end{bmatrix},$$
   where $D$ is a diagonal matrix with diagonal entries $r_1,\ldots,r_{\ell}$. Our goal is to compute the determinant of $L$ with the first row and column removed; let this matrix be $L_1$. If we subtract $\frac{1}{n-k}$ times the first $\ell-1$ rows from all of the last $n-k$ rows of $L_1$, we obtain a matrix where all entries in the first $\ell-1$ columns, except those in the diagonal, are $0$. Thus the determinant is equal to the product of these diagonal entries $r_2(n-k),\ldots,r_{\ell}(n-k)$ times the determinant of a matrix of the block form
   $$\begin{bmatrix} (n-r) I - E & -E \\ -E & n I - \big( 1 + \tfrac{r-r_1}{n-k} \big)E \end{bmatrix},$$
   where the blocks have length $k-r$ and $n-k$ respectively. This matrix has $n-r$ as an eigenvalue of multiplicity $k-r-1$, since subtracting $n-r$ times the identity yields a matrix with $k-r$ identical rows. For the same reason, $n$ is an eigenvalue of multiplicity $n-k-1$. It remains to determine the remaining two eigenvalues. The corresponding eigenvectors can be constructed as follows: let the first $k-r$ entries (corresponding to the first block) be equal to $a$, and the remaining entries equal to $b$. It is easy to verify that this becomes an eigenvector for the eigenvalue $\lambda$ if the simultaneous equations
   \begin{align*}
   (n-k)a - (n-k)b &= \lambda a, \\
   -(k-r)a + (k-r+r_1)b &= \lambda b,
   \end{align*}
   are satisfied. The two solutions are the eigenvalues of the $2 \times 2$-coefficient matrix of this system, and their product is the determinant of this $2 \times 2$ matrix, which is
   $$(n-k)(k-r+r_1) - (n-k)(k-r) = (n-k)r_1.$$
   It finally follows that the determinant of $L_1$, thus the number of spanning trees of $K''$, is equal to
   \begin{multline*}
   r_2(n-k) \cdots r_{\ell}(n-k) \cdot (n-r)^{k-r-1} n^{n-k-1} (n-k)r_1 \\
   = n^{n-k-1}  (n-r)^{k-r-1} (n-k)^{\ell} r_1\cdots r_{\ell}.
   \end{multline*}
   Multiplying by $r_1^{r_1-2} \cdots r_{\ell}^{\ell-2}$ (the number of possibilities for the spanning trees induced in the components $R_1,\ldots,R_{\ell}$), we obtain the desired formula.
   \end{proof}

As a consequence of Lemma~\ref{claim:tree-enumeration} for $\ell = 1$,
the probability $\P(R_{n}^{(k)} = r)$ can be obtained by
multiplying $n^{n-k-1} (n-r)^{k-r-1} (n-k) r^{r-1}$ with $r \binom{k}{r}$
(which gives the number of rooted labeled trees on $n$ vertices whose
root is contained in a cluster of size $r$ among the first $k$
uncovered vertices),
and then normalizing by $n^{n-1}$, the number of labeled rooted trees on $n$ vertices.

\begin{theorem}\label{thm:R_Prob_labeled}
The exact distribution of $R_{n}^{(k)}$ is characterized by the following probability mass function (p.m.f.), which is given by the following formula for $0 \le m \le k \le n$ and $n \ge 1$ (and is equal to $0$ otherwise):
\begin{equation*}
  \P(R_{n}^{(k)} = m) =
	\begin{cases}
	  1-\frac{k}{n}, & \quad \text{for $m=0$},\\
		\frac{m^{m} (n-k) (n-m)^{k-m-1}}{n^{k}} \binom{k}{m}, & \quad \text{for $1 \le m \le k < n$},\\
		1, & \quad \text{for $m=k=n$}.
	\end{cases}
\end{equation*}

Depending on the growth of $k=k(n)$, we obtain the following limiting behavior:
\begin{itemize}
  \item $k$ small, i.e., $k = o(n)$:
	\begin{equation*}
    R_{n}^{(k)} \xrightarrow{p} 0.
  \end{equation*}
  \item $k$ in central region, i.e., $k \sim \alpha n$ with $0 < \alpha < 1$:
	\begin{gather*}
	  R_{n}^{(k)} \xrightarrow{d} R_{\alpha}, \quad \text{where the discrete r.v.\ $R_{\alpha}$ is characterized by its p.m.f.}\\
	  \P(R_{\alpha} = m) =: p_{m} =
		\begin{cases} 1-\alpha, & \quad m = 0,\\ \frac{m^{m}}{m!} (1-\alpha) \alpha^{m} e^{-\alpha m}, & \quad m \ge 1, \end{cases}
	\end{gather*}
or alternatively by the probability generating function $p(v) = \sum_{m \ge 0} p_{m} v^{m} = \frac{1-\alpha}{1-T^{\bullet}(v \alpha e^{-\alpha})}$.
	\item $k$ subcritically large, i.e., $k=n-d$ with $d=\omega(\sqrt{n})$ and $d=o(n)$:
	\begin{gather*}
	  \Big(\frac{d}{n}\Big)^{2} \cdot R_{n}^{(k)} \xrightarrow{d} \text{\textsc{Gamma}}\Big(\frac{1}{2},\frac{1}{2}\Big),
	\end{gather*}
where $\text{\textsc{Gamma}}(\frac{1}{2},\frac{1}{2})$ is a Gamma-distribution characterized by its density $f(x) = \frac{1}{\sqrt{2 \pi x}} \, e^{-\frac{x}{2}}$, for $x > 0$.
	\item $k$ critically large, i.e., $k=n-d$ with $d \sim c \sqrt{n}$ and $c > 0$:
	\begin{gather*}
	  \frac{1}{n} \cdot R_{n}^{(k)} \xrightarrow{d} R(c),
	\end{gather*}
where the continuous r.v.\ $R(c)$ is characterized by its density $f_{c}(x) = \frac{1}{\sqrt{2 \pi}} \frac{c}{\sqrt{x} (1-x)^{\frac{3}{2}}} \, e^{-\frac{c^{2} x}{2(1-x)}}$, for $0 < x < 1$.
	\item $k$ supercritically large, i.e., $k=n-d$ with $d=\omega(1)$ and $d = o(\sqrt{n})$:
	\begin{gather*}
	  \frac{1}{d^{2}} \cdot \big(n - R_{n}^{(k)}\big) \xrightarrow{d} D,
	\end{gather*}
where the continuous r.v.\ $D$ is characterized by its density $f(x) = \frac{1}{\sqrt{2 \pi} \, x^{\frac{3}{2}}} \, e^{-\frac{1}{2x}}$, $x > 0$.
	\item $k$ supercritically large with fixed difference, i.e., $k=n-d$ with $d$ fixed:
	\begin{gather*}
	  n - d - R_{n}^{(k)} \xrightarrow{d} D(d),\\
		\intertext{where the discrete r.v.\ $D(d)$ is characterized by the p.m.f.}
		\P(D(d) = j) =: p_{j} = e^{-d} \cdot \frac{d (d+j)^{j-1}}{j!} \cdot e^{-j}, \quad j \ge 0,
	\end{gather*}
or alternatively via the probability generating function $p(v) = \sum_{j \ge 0} p_{j} v^{j} = e^{d (T^{\bullet}(\frac{v}{e})-1)}$.
\end{itemize}
\end{theorem}
\begin{proof}
   The probability mass function of $R_{n}^{(k)}$ follows from the considerations
   made before the statement of the theorem. Due to its explicit nature, the limiting distribution results stated in Theorem~\ref{thm:R_Prob_labeled} can be obtained in a rather straightforward way by applying Stirling's formula for the factorials after distinguishing several cases.
\end{proof}

\begin{remark}
  Of course, for labeled trees, the distribution of $R_{n}^{(k)}$ matches with the distribution of the cluster size of a random vertex. Furthermore, by conditioning, one can easily transfer the results of Theorem~\ref{thm:R_Prob_labeled} to results for the size $S_{n}^{(k)}$ of the cluster of the $k$-th uncovered vertex: $\P(S_{n}^{(k)} = m) = \P(R_{n}^{(k)} = m | R_{n}^{(k)} > 0) = \frac{n}{k} \cdot \P(R_{n}^{(k)}=m)$, for $m \ge 1$.
\end{remark}

\section{Size of the largest uncovered component}\label{sec:largest-component}

With knowledge about the behavior of the root cluster at our disposal, we
return to non-rooted labeled trees and study the size of the
largest cluster. To this aim, we introduce the random variable $X_{n, r}^{(k)}$
which models the number of components of size $r$ after uncovering the
vertices $1$ to $k$ in a uniformly random labeled tree of size $n$.

Formally, $X_{n, r}^{(k)}\colon \mathcal{T}_n \to \Z_{\geq 0}$.
Note that we have, for all labeled trees $T\in \mathcal{T}_n$,
\begin{equation}\label{eq:component-sum}
   \sum_{r=1}^n r\cdot X_{n, r}^{(k)}(T) = k.
\end{equation}

\begin{theorem}\label{thm:expected-number-of-clusters}
   Let $n, k, r\in \Z_{\geq 0}$ with $0\leq r \leq k \leq n$. The
   expected number of connected components of size $r$ after uncovering
   $k$ vertices of a labeled tree of size $n$ chosen uniformly at random
   is
   \begin{equation}\label{eq:expected-components}
      \E X_{n,r}^{(k)} = \binom{k}{r} \Big(\frac{r}{n}\Big)^{r-1}
         \Bigl(1 - \frac{k}{n}\Bigr) \Bigl(1 - \frac{r}{n}\Bigr)^{k-r-1}.
   \end{equation}
\end{theorem}
\begin{proof}
   Observe that $X_{n,r}^{(k)}$ can be written as a sum of Bernoulli random variables
   \[
      X_{n, r}^{(k)} = \sum_{\substack{S \subseteq [k] \\ \abs{S} = r}} X_{n, S}^{(k)},
   \]
   with $X_{n, S}^{(k)}$ being $0$ or $1$ depending on whether or not the vertices in $S$ form a cluster
   after $k$ uncover steps.
   By symmetry and linearity of the expected value, we have
   \[
      \E X_{n, r}^{(k)} = \sum_{\substack{S \subseteq [k] \\ \abs{S} = r}} \E X_{n, S}^{(k)}
      = \binom{k}{r} \E X_{n, [r]}^{(k)}.
   \]
   A formula for the expected value on the right-hand side follows from
   Lemma~\ref{claim:tree-enumeration}, and thus proves the theorem.
\end{proof}

In the spirit of the observation in~\eqref{eq:component-sum}, the
formula in Lemma~\ref{claim:tree-enumeration} provides
a combinatorial proof for the following summation identity.

\begin{corollary}\label{cor:cluster-sum}
   Let $n, k \in \Z_{\geq 0}$ with $0\leq k\leq n$. Then,
   the identity
   \begin{equation}\label{eq:cluster-sum}
      \sum_{r = 1}^{k} \binom{k}{r} r^r n^{n-k-1} (n-r)^{k-r-1} (n-k) = k n^{n-2}
   \end{equation}
   holds.
\end{corollary}
\begin{proof}
   The right-hand side enumerates the vertices in $[k]$ in all labeled trees
   on $n$ vertices. The left-hand side does the same, with the summands enumerating
   the vertices in connected components of size $r$.
\end{proof}

\begin{remark}
   Observe that the identity in~\eqref{eq:cluster-sum} can be rewritten as
   \[
      \sum_{r = 1}^{k} \binom{k}{r} r^r (n-r)^{k-r-1} = \frac{k}{n-k} n^{k-1},
   \]
   which is a specialized form of Abel's Binomial Theorem---a classical, and
   well-known result; see, e.g., \cite{Riordan:1968:combinatorial-identities}.
\end{remark}

For a tree $T\in\mathcal{T}$, let $\cmax{k}(T)$ denote the largest
connected component of $T$ after uncovering the first $k$ vertices.

\begin{theorem}
   Let $n\in\Z_{\geq 0}$, and let $T_n\in\mathcal{T}_n$ be
   a tree chosen uniformly at random. Then the behavior of the random variable
   $\cmax{k}(T_n)$ as $n\to\infty$ can be described as follows:
   \begin{itemize}
   \item for $k = n-d$ with $d = \omega(\sqrt{n})$ (\emph{subcritical case}),
   we have $\cmax{k}(T_n)/n \xrightarrow{p} 0$.
   \item for $k = n-d$ with $d = o(\sqrt{n})$ (\emph{supercritical case}), we have $\cmax{k}(T_n)/n \xrightarrow{p} 1$.
   With high probability, there is one ``giant'' component whose size is asymptotically equal to $n$.
   \end{itemize}
\end{theorem}
\begin{proof}
   For the subcritical case, we use the expected root cluster size
   from Theorem~\ref{thm:R_Exp_labeled}. Since a cluster of size $r$ contains the root with probability $\frac{r}{n}$, we have
   \begin{align*}
      \frac{n^2}{d^2} &\sim \E R_{n}^{(n-d)} = \sum_{r = 0}^{n-d} \E X_{n,r}^{(n-d)} \cdot r \cdot \frac{r}{n}
      \geq \sum_{r = m}^{n-d} \E X_{n,r}^{(n-d)} \frac{r^2}{n}
      \geq \frac{m^2}{n} \sum_{r=m}^{n-d} \E X_{n,r}^{(n-d)}\\
      &\geq \frac{m^2}{n} \P(\cmax{n-d}(T_n) \geq m).
   \end{align*}
   This implies that
   \[
      \P(\cmax{n-d}(T_n) \geq m) = \Oh\Bigl(\frac{n^3}{d^2 m^2}\Bigr),
   \]
   so if $m = \epsilon n$ for any fixed $\epsilon > 0$, we have $\P(\cmax{n-d}(T_n) \geq m) \to 0$.

   In the supercritical case, we recall the corresponding case for
   the size of the root cluster from Theorem~\ref{thm:R_Exp_labeled}. Using Markov's
   inequality yields, for any $\epsilon > 0$,
   \[ \P(n - R_n^{(k)} \geq \epsilon n) \leq \frac{n - \E(R_{n}^{(k)})}{\epsilon n} \sim \frac{d \sqrt{n}}{\epsilon n} \xrightarrow{n\to\infty} 0. \]
   Thus, the root cluster is the largest cluster of size $\sim n$ with high probability.
   Translating this from rooted to unrooted trees proves the theorem.
\end{proof}

In the critical case where $n - k \sim c\sqrt{n}$ for a constant $c$,
we are also able to characterize the behavior of $\cmax{k}(T_n)/n$:
this variable converges weakly to a continuous
limiting distribution. In order to describe the distribution, we
first require the following lemma.

\begin{lemma}\label{lem:no_doubled}
Suppose that $n - k \sim c \sqrt{n}$, and fix $\alpha > 0$. The probability that the forest induced by the vertices with labels in $[k]$ of a uniformly random labeled tree $T_n$ with $n$ vertices contains two components, each with at least $\alpha n$ vertices, whose size is equal, goes to $0$ as $n \to \infty$.
\end{lemma}
\begin{proof}
Suppose that there are two components with $a$ vertices each, where $a \geq \alpha n$. By Lemma~\ref{claim:tree-enumeration}, the probability for this to happen is given by
$$P(a) = \frac{\binom{k}{a,a,k-2a} n^{n-k-1}(n-2a)^{k-2a-1} (n-k)^2 a^{2a-2}}{n^{n-2}},$$
provided that $k \geq 2a$ (otherwise, the probability is trivially $0$). The initial multinomial coefficient gives the number of ways to choose the labels of the two components, the denominator is simply the total number of labeled trees. Our aim is to estimate this expression. The case $k = 2a$ is easy to deal with separately, so assume that $k > 2a$. Then by Stirling's formula we have, for some constant $C_1$,
\begin{align*}
P(a) &\leq \frac{C_1 k^{k+1/2}}{a^{2a+1}(k-2a)^{k-2a+1/2}} \cdot \frac{n^{n-k-1}(n-2a)^{k-2a-1} (n-k)^2 a^{2a-2}}{n^{n-2}} \\
&= \frac{C_1 k^{1/2}(n-k)^2n}{a^3(n-2a)(k-2a)^{1/2}} \Big( 1 + \frac{n-k}{k} \Big)^{-k} \Big( 1 + \frac{n-k}{k-2a} \Big)^{k-2a}.
\end{align*}
It is well known that $(1+\beta/x)^x$ is increasing in $x$ for fixed $\beta$. So if $k - 2a \leq n^{3/4}$, we have, using the assumption that $n-k \sim c \sqrt{n}$,
\begin{align*}
\Big( 1 &+ \frac{n-k}{k} \Big)^{-k} \Big( 1 + \frac{n-k}{k-2a} \Big)^{k-2a} \\ &\leq \Big( 1 + \frac{n-k}{k} \Big)^{-k} \Big( 1 + \frac{n-k}{n^{3/4}} \Big)^{n^{3/4}} \\
&= \exp \Big({-k} \log \Big( 1 + \frac{n-k}{k} \Big) + n^{3/4} \log \Big( 1 + \frac{n-k}{n^{3/4}} \Big) \Big) \\
&= \exp \Big({-k} \Big( \frac{n-k}{k} + \Oh(n^{-1}) \Big) + n^{3/4} \Big( \frac{n-k}{n^{3/4}} - \frac{(n-k)^2}{2n^{3/2}} + \Oh(n^{-3/4}) \Big)\Big) \\
&= \exp \Big( {-c^2} n^{1/4} + o(n^{1/4}) \Big).
\end{align*}
In this case, $P(a)$ goes to $0$ faster than any power of $n$. Otherwise, i.e., if $k - 2a > n^{3/4}$, we have $k-2a \sim n-2a$, and using the assumption that $a \geq \alpha n$ as well as the same Taylor expansion as above, we obtain
$$P(a) \leq \frac{C_2(n-k)^2}{n^{3/2}(n-2a)^{3/2}} \exp\Big({- \frac{(n-k)^2}{2(n-2a)} }\Big)$$
for a constant $C_2$. We can rewrite this as
$$P(a) \leq \frac{C_2}{n^{3/2}(n-k)} f \Big( \frac{(n-k)^2}{n-2a} \Big)$$
with $f(x) = x^{3/2} e^{-x/2}$. Since this function is bounded, we have proven that $P(a) = \Oh(n^{-2})$, uniformly in $a$. Summing over all possible values of $a$, it follows that the probability in question is $\Oh(n^{-1})$. In particular, it goes to $0$.
\end{proof}

Now we are able to prove the following description of the limiting
distribution of $\cmax{k}(T_n)/n$.

\begin{theorem}
   Suppose that $n - k \sim c \sqrt{n}$, and fix $\alpha > 0$. The probability that the forest induced by the vertices with labels in $[k]$ of a uniformly random labeled tree $T_n$ with $n$ vertices contains a component with at least $\alpha n$ vertices tends to
   \begin{equation*}
   \sum_{j \geq 1} \frac{(-1)^{j-1} c^j}{(2\pi)^{j/2}} \idotsint\limits_{\substack{\alpha \leq t_1 < \cdots < t_j \\ \tau_j = t_1 + \cdots + t_j < 1}} \prod_{i=1}^j t_i^{-3/2} (1-\tau_j)^{-3/2} \exp \Big( - \frac{c^2\tau_j}{2 (1-\tau_j)} \Big) \,dt_1 \cdots dt_j
   \end{equation*}
   as $n \to \infty$.
\end{theorem}
\begin{proof}
   Let $r_1,\ldots,r_\ell$ be positive integers with $\alpha n \leq r_1 < \cdots < r_\ell$ and $r_1+\cdots+r_\ell \leq k$. By Lemma~\ref{claim:tree-enumeration}, the probability that the forest induced by vertices with labels in $[k]$ has components of sizes $r_1,\ldots,r_\ell$ is given by the following formula, with $r = r_1 + \cdots + r_{\ell}$:
   $$P(r_1,\ldots,r_{\ell}) = \frac{\binom{k}{r_1,\ldots,r_{\ell},k-r} n^{n-k-1}(n-r)^{k-r-1} (n-k)^{\ell} r_1^{r_1-1}\cdots r_{\ell}^{r_{\ell}-1}}{n^{n-2}},$$
   and the same argument as in Lemma~\ref{lem:no_doubled} shows that this probability is $\Oh(n^{-\ell})$, uniformly in $r_1,\ldots,r_\ell$. Moreover, if we set $r_i = t_i n$, Stirling's formula gives us the following asymptotic formula for this probability after some manipulations: with $t = t_1 + \cdots + t_{\ell}$, it is
   $$P(r_1,\ldots,r_{\ell}) \sim \frac{c^\ell}{n^{\ell} (2\pi)^{\ell/2}} \prod_{i=1}^\ell t_i^{-3/2} (1-t)^{-3/2} \exp \Big( - \frac{c^2t}{2 (1-t)} \Big).$$
   Moreover, by the inclusion-exclusion principle, the probability that there is at least one component of size at least $\alpha n$ is given by
   \begin{multline*}
   \sum_{\alpha n \leq r_1 \leq k} P(r_1) - \sum_{\substack{\alpha n \leq r_1 < r_2 \\ r_1+r_2 \leq k}} P(r_1,r_2) + \cdots \\
   + (-1)^{j-1} \sum_{\substack{\alpha n \leq r_1 < \cdots < r_j \\ r_1 + \cdots + r_j \leq k}} P(r_1,\ldots,r_j) + \cdots + \Oh(n^{-1}).
   \end{multline*}
   The final error term takes the possibility into account that there are two components of the same size. The probability of this event is $\Oh(n^{-1})$ by Lemma~\ref{lem:no_doubled}. Note that we actually only need a finite number of terms, as the sums become empty for $j \alpha > 1$. If we plug in the asymptotic formula for $P(r_1,\ldots,r_{\ell})$ and pass to the limit, the sums become integrals, and we obtain the desired formula.
\end{proof}

\section*{Acknowledgment}

We would like to thank Svante Janson for pointing out a gap in the proof of Theorem~\ref{thm:edge-process} in the extended abstract of this paper.

\bibliographystyle{abbrv}
\bibliography{clusters}

\end{document}